\documentclass[journal]{IEEEtran}
\usepackage[short]{optidef}
\usepackage{graphicx, float}
\usepackage[subtle]{savetrees}
\usepackage{amsmath}
\usepackage{siunitx}
\usepackage{amssymb}
\usepackage{amsthm}
\usepackage{booktabs}
\usepackage{cite}
\usepackage{enumerate}

\usepackage[final]{changes}
\usepackage{changes}
\definechangesauthor[name=Nawaf, color=blue]{1}
\ifCLASSOPTIONcompsoc
    \usepackage[caption=false, font=normalsize, labelfont=sf, textfont=sf]{subfig}
\else
\usepackage[caption=false, font=footnotesize]{subfig}
\fi
\newtheorem{theorem}{Theorem}[section]
\newtheorem{corollary}{Corollary}[theorem]

\newtheorem*{remark}{Remark}

\IEEEoverridecommandlockouts  

\begin{document}
%
\title{Optimal multi-period dispatch of distributed energy resources in unbalanced distribution feeders}
%
%
%

\author{Nawaf~Nazir,~\IEEEmembership{Student~Member,~IEEE,}
        ~Pavan~Racherla,~
        and~Mads~Almassalkhi,~\IEEEmembership{Senior Member,~IEEE}
\thanks{Authors are with the Department
of Electrical and Biomedical Engineering, University of Vermont, Burlington,
VT, 05405 USA e-mail: mnazir@uvm.edu
}
\thanks{This work  was  supported  by  the  U.S.  Department  of  Energy’s  Office  of Energy Efficiency and Renewable Energy (EERE) award DE-EE0008006.}}

%
%



\maketitle

\begin{abstract}
This paper presents an efficient algorithm for the multi-period optimal dispatch of deterministic inverter-interfaced energy storage in an unbalanced distribution feeder with significant solar PV penetration. The three-phase, non-convex loss-minimization problem is formulated as a convex second-order cone program (SOCP) for the dispatch of batteries in a receding-horizon fashion in order to counter against the variable, renewable generation. The solution of the SOCP is used to initialize a nonlinear program (NLP) in order to ensure a physically realizable solution. The phenomenon of simultaneous charging and discharging of batteries is rigorously analyzed and conditions are derived that guarantee it is avoided. Simulation scenarios are implemented with GridLab-D for the IEEE-13 and IEEE-123 node test feeders and illustrate not only AC feasibility of the solution, but also near-optimal performance and solve-times within a minute. 
\end{abstract}


%
\IEEEpeerreviewmaketitle

\section{Introduction}\label{introduction}

The rapid growth in distributed solar PV generation over the past decade has prompted significant interests and investments in demonstration of substation automation technology, distributed energy resources or DERs, such as energy storage and smart inverters, and autonomous demand response~\cite{singer2010enabling,ackermann2017paving}.  
To maintain grid operating conditions under significant renewable (and intermittent) generation, the utility grid operators can leverage the power and energy flexibility inherent to many DERs. However, unlike traditional generation, DERs, such as batteries are energy constrained, which give rise to the need for multi-period decision-making and predictive optimization. Innovative energy service providers, such as ConEdison of NY are moving towards so-called distributed system platforms (DSPs), which represent innovative business models for holistic management of DERs~\cite{conED}. That is, concepts such as DSPs allow for distributed, layered-decomposition optimization schemes that can aggregate and dispatch DERs as virtual batteries (VBs) in a bottom-up fashion to provide energy services at different spatio-temporal scales~\cite{kristov2016tale}. 

The optimal power flow (OPF) is a useful tool to coordinate the grid resources subject to the nonlinear power flow equations and network constraints~\cite{carpentier1962contribution}. 
For constant power loads, the AC power flow equations relate the voltages in the network with the power injections. It has been shown in~\cite{wang2018chordal} that the solution space of the three-phase OPF is non-convex and the solution space of the OPF problem and its convex hull are different. \added{However, many works in literature such as~\cite{karagiannopoulos2018centralised} have shown the importance of considering the  imbalances in a distribution network for accurate analysis.} The full ACOPF model represents an NP-hard, non-convex problem.  Recently, there have been efforts to use convex optimization techniques to solve the OPF problem~\cite{lavaei2012zero,lavaei2014geometry}. Previous works in literature have shown that for certain (e.g. radial) network toplogies, the convex relaxations, such as second order cone programs (SOCP) and semi-definite programs (SDP) can be exact~\cite{jabr2006radial,farivar2011inverter}.
 Traditionally, \textit{DistFlow} algorithms based on \textit{branch flow power models}, \added{which is an exact nonlinear formulation of the distribution network power flow equations,} are used to solve the OPF problem in distribution networks~\cite{baran1989optimal}. However, these methods only consider balanced single phase equivalent models. Distribution networks are inherently unbalanced which makes it important to study the full three-phase models of these networks~\cite{katiraei2015no}.

  Linear approximate models can also be powerful when they are sufficiently accurate. One particular approximation is an extension of the \textit{DistFlow} model~\cite{baran1989optimal} to unbalanced power flows, \textit{LinDist3Flow},\added{ which is a linear formulation of the branch flow applied to 3-phase distribution networks.} This model is obtained by linearization and assumptions of fixed per-phase imbalances~\cite{sankur2016linearized,robbins2016optimal}. Linear models, even though simple and computationally efficient, do not provide guarantees on optimality and feasibility, which are important for scheduling/dispatch problems, such as battery dispatch.

 In~\cite{dall2013distributed,gan2014convex} the authors use SDP rank constraint relaxation to the three phase model of a distribution network, whereas in~\cite{gayme2013optimal,gopalakrishnan2013global}, the authors have used multi-period SDP relaxation techniques to solve this problem for transmission networks, however, SDP solvers are still not numerically robust~\cite{roux2016validating}. 
 In~\cite{marley2017solving}, the authors have utilized a single phase OPF AC-QP algorithm that is initialized with an SOCP relaxation. However, this multi-period OPF formulation neglects the non-unity charge and discharge efficiency of the battery, which can create solutions to the OPF problem that are physically unrealizable due to simultaneous charging and discharging of batteries.
 
In~\cite{nazir2018receding}, we developed a three-phase convex SOCP relaxation of the multi-period OPF problem. We also provided sufficient conditions to avoid simultaneous charging and discharging of batteries in distribution networks with non-unity charging and discharging efficiencies. In this paper, we extend the work in~\cite{nazir2018receding} by developing a multi-period SOCP-NLP algorithm that provides a near optimal but guaranteed feasible solution. We also extend and generalize the analysis on simultaneous charging and discharging to different objective functions and provide comprehensive simulation results on 100+ node feeder system to illustrate computational effectiveness of the proposed optimization algorithms. The optimized solutions obtained from the relaxed SOCP model, are used to initialize a nonlinear program (NLP) of the actual AC power flow to obtain a physically realizable solution.
Furthermore, the real-power solutions that form the energy trajectory and are obtained from the SOCP are fixed in the NLP, leading to a decoupling of the different time-steps. As a result, the NLP solves each time-step separately (and possibly in parallel), leading to a scalable framework. Validation is performed with GridLab-D~\cite{chassin2008gridlab}. 

Thus,  main contributions of this paper are as follows:
\begin{enumerate}
  \item A novel approach to obtain a near-optimal feasible solution by temporal de-coupling of the NLP initialized with the solution from a multi-period three-phase SOCP convex relaxation is presented. By fixing the real-power solutions, the time-steps of the NLP are decoupled leading to a scalable framework
        \item Rigorous analysis is performed on the convex formulation and general conditions are derived that guarantee that the phenomenon of simultaneous charging and discharging of batteries is avoided for different types of network objectives.
    
\end{enumerate}

Section~\ref{conv_form} develops the three-phase OPF problem and the convex formulation for the dispatch of batteries to minimize the network line losses. The role of the objective function on the conditions for which simultaneous charging and discharging of batteries is avoided are analyzed in Section~\ref{comp_const}. Section~\ref{SOCP_NLP} guarantees a physically realizable battery multi-period dispatch by coupling the relaxed SOCP with the exact NLP formulation. Simulation-based analysis and validation results are discussed in Section~\ref{sim_results} for the IEEE-123 node system and GridLab-D. Conclusions and future research directions are presented in Section~\ref{conclusion}.

\section{Convex formulation of Multi-period 3-phase OPF}\label{conv_form}
The aim of this section is to develop a convex formulation of the multi-period optimal power flow in three-phase distribution networks \added{that can be used for the dispatch of DERs in the network. Figure~\ref{fig_stl} illustrates the types of DERs available to the optimizer at each node and the corresponding notation}.
A common objective in distribution networks is to minimize the real power losses, while keeping the system within its operational grid constraints~\cite{atwa2010optimal}. This program optimizes the batteries (i.e., the real and reactive powers) in the network, whose architecture is shown in figure \ref{fig_stl}, over the minute-to-minute time-scale. Such fast solution times for large networks requires formulation that can be solved in polynomial time. Thus, we focus on the following convex formulation.
\begin{figure}[!ht]
\centering
\includegraphics[width=2.8in]{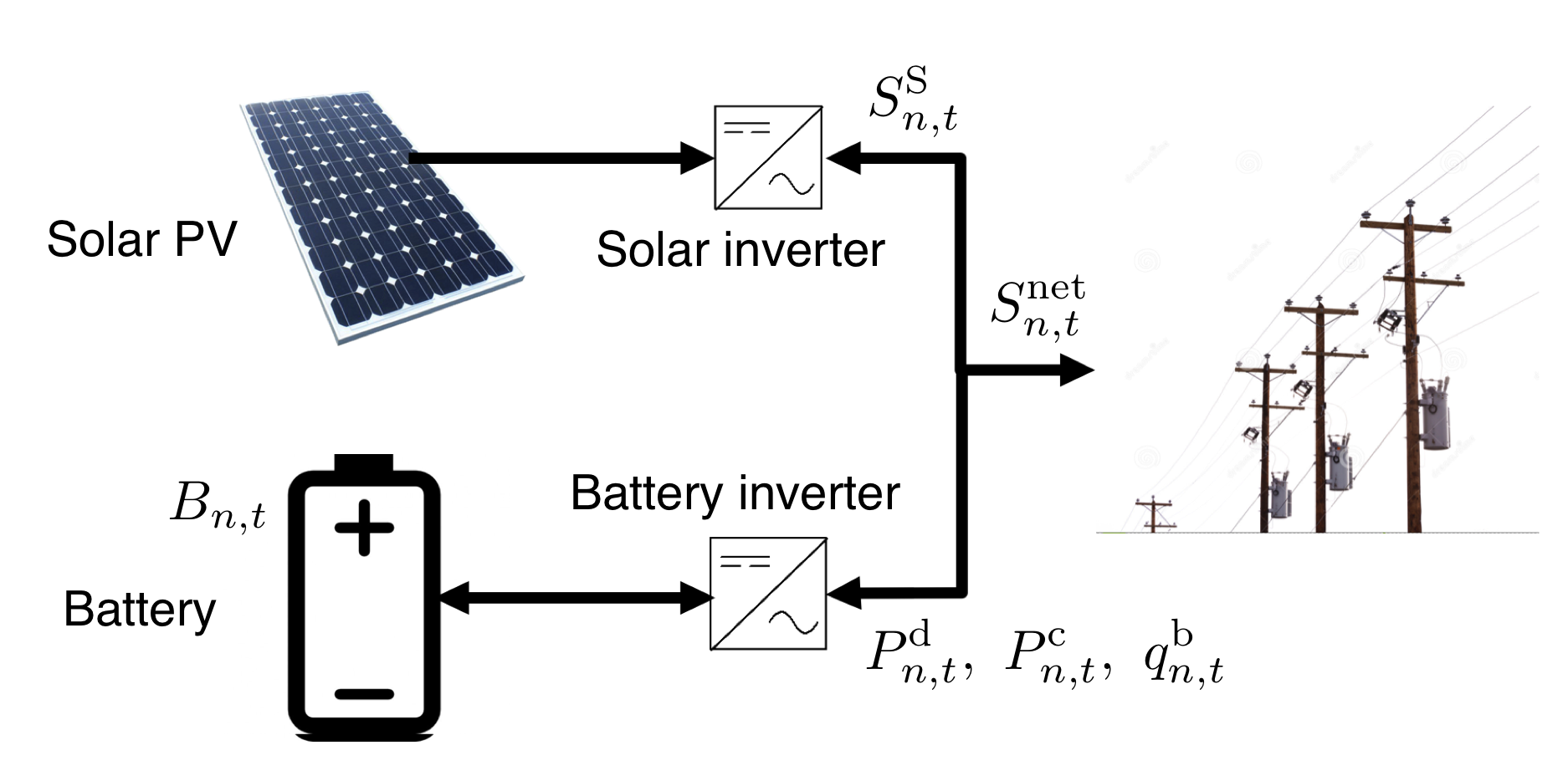}
\caption{Distributed storage architecture. The batteries are controlled through a four quadrant control scheme and can supply and consume both real and reactive power. Each distributed storage is composed of a renewable source of energy such as solar power and a battery bank, each with its own inverter.}
\label{fig_stl}
\end{figure}

A three-phase second order cone program (SOCP) is developed to solve the multi-period optimization problem. A branch flow model (BFM) is used to represent the power flow equations of the three-phase network. 

\subsection{Mathematical and modeling notation}\label{math_notation}
Consider a radial distribution network with $n$ nodes, where $\mathcal{N} = \{1, 2, . . . , n\}$ is the set of all nodes, $\phi = \{a, b, c\}$ is the set of phases at each node, $\mathcal{L}=\{1, 2, . . . , l\}= \{(m,n)\}\subset (\mathcal{N}\times \mathcal{N})$ is the set of all branches, $\mathcal{G}=\{1,2,. . . ,g\}$ is the set of all nodes with DERs and $\mathcal{T}=\{0, . . . , T-1\}$ be the prediction horizon. Let vector $V_{n,t} \in \mathbb{C}^{|\phi|}$ be the voltage at node $n$ and time $t$, with $W_{n,t}=V_{n,t}V_{n,t}^*$, $i_{l,t} \in \mathbb{C}^{|\phi|}$ be the current in branch $l$ at time $t$, with $I_{l,t}=i_{l,t}i_{l,t}^*$, $S_{l,t}=V_{n,t}i_{l,t}^*$ be the apparent power in branch $l$ at time $t$ and $Z_l=R_l+jX_l \in \mathbb{C}^{|\phi| \times |\phi|}$ be the impedance of branch $l$. Let $S^{\text{net}}_{n,t} \in \mathbb{C}^{|\phi|}$ be the apparent power injection , $S^{\text{L}}_{n,t} \in \mathcal{C}^{|\phi|}$ be the apparent load, $S^{\text{S}}_{n,t} \in \mathcal{C}^{|\phi|}$ be the apparent power from solar PV, $P^{\text{c}}_{n,t} \in \mathcal{R}^{|\phi|}$ and $P^{\text{d}}_{n,t} \in \mathcal{R}^{|\phi|}$ be the charge and discharge power of battery, $q^{\text{b}}_{n,t}\in \mathcal{R}^{|\phi|}$ be the reactive power from battery and $B_{n,t} \in \mathcal{R}^{|\phi|}$ be the state of charge (SoC) of the battery, all defined at node $n$ and time $t$, assuming the nodes have Wye-connected single phase batteries. The symbols $\circ$, $(.)^\ast$ and $\text{diag}(.)$ represent the Hadamard product of matrices, the complex conjugate operator, and the diagonal operator, respectively.

\subsection{Mathematical Formulation}\label{math_model}

If $x=\{P^{\text{d}}_{n,t},P^{\text{c}}_{n,t},q^{\text{b}}_{n,t},S^{\text{S}}_{n,t}\}$ be the set of independent optimization variables $\forall \ t \in \mathcal{T}, \ n \in \mathcal{N}$, then the problem of optimally dispatching the batteries to minimize objective function $f(x)$ can be formulated as:
\begin{mini!}[3]
{P^{\text{d}}_{n,t},P^{\text{c}}_{n,t},q^{\text{b}}_{n,t},S^{\text{S}}_{n,t}}{f(x)\label{eq:P1_obj}}
{\label{eq:P1}}{}
\addConstraint{\begin{bmatrix}
W_{n,t} & S_{l,t}\\
S_{l,t}^{*} & I_{l,t}\end{bmatrix}}{=
\begin{bmatrix}
V_{n,t}\\
i_{l,t}\end{bmatrix}
\begin{bmatrix}
V_{n,t}\\
i_{l,t}\end{bmatrix}^{*} \qquad \forall l\in \mathcal{L}\label{eq:P1_BFM}}
\addConstraint{0}{=W_{n,t}-W_{m,t}+(S_{l,t}Z_l^*+Z_lS_{l,t})-Z_lI_{l,t}Z_l^*\ \ \forall l\in \mathcal{L}\label{eq:P1_volt_rel}}
\addConstraint{0}{=\text{diag}(S_{l,t}-Z_l I_{l,t}-\sum_{p}S_{p,t})+S^{\text{net}}_{n,t}\ \ \ \forall l\in \mathcal{L}\label{eq:P1_power_balance}}
\addConstraint{0}{=\text{real}(S^{\text{net}}_{n,t}-S^{\text{S}}_{n,t}+S^{\text{L}}_{n,t})-P^{\text{d}}_{n,t}+P^{\text{c}}_{n,t} \qquad \forall n\in \mathcal{G}\label{eq:P1_node_real_balance}}
\addConstraint{0}{=\text{imag}(S^{\text{net}}_{n,t}-S^{\text{S}}_{n,t}+S^{\text{L}}_{n,t})-q^{\text{b}}_{n,t} \qquad \forall n\in \mathcal{G}\label{eq:P1_node_reactive_balance}}
\addConstraint{|\text{diag}(S_{l,t})|}{\leq S_{\text{max},l} \qquad \forall l\in \mathcal{L}\label{eq:P1_line_const}}
\addConstraint{V_{\text{min},n}^2\leq \text{diag}(W_{n,t})}{\leq V_{\text{max},n}^2 \ \forall n\in \mathcal{N}\label{eq:P1_volt_const}}
\addConstraint{|S^{\text{S}}_{n,t}|}{\leq G_{\text{max},n} \qquad \forall n\in \mathcal{G}\label{eq:P1_solar_inv_limit}}
\addConstraint{(P^{\text{d}}_{n,t}-P^{\text{c}}_{n,t})^2+(q^{\text{b}}_{n,t})^2}{\leq H_{\text{max},n}^2, \qquad \forall n\in \mathcal{G}\label{eq:P1_battery_inv_limit}}
\addConstraint{0}{=B_{n,t+1}-B_{n,t}-\eta_{\text{c},n}P^{\text{c}}_{n,t}\Delta t+\frac{P^{\text{d}}_{n,t}}{\eta_{\text{d},n}}\Delta t \qquad \forall n\in \mathcal{G}\label{eq:P1_battery_power_rel}}
\addConstraint{B_{\text{min},n}\leq B_{n,t}}{\leq B_{\text{max},n} \qquad \forall n\in \mathcal{G}\label{eq:P1_SOC_limit}}
\addConstraint{ 0\leq P^{\text{d}}_{n,t}}{\leq P_{\text{max},n} \qquad \forall n\in \mathcal{G}\label{eq:P1_Pd_limit}}
\addConstraint{ 0\leq P^{\text{c}}_{n,t}}{\leq P_{\text{max},n} \qquad \forall n\in \mathcal{G}\label{eq:P1_Pc_limit}}
\addConstraint{P^{\text{d}}_{n,t}\circ P^{\text{c}}_{n,t}}{=0 \qquad \forall n\in \mathcal{G}\label{eq:P1_comp_const}}
\end{mini!}

where the above equations hold $\forall t \in \mathcal{T}$. In the optimization problem \eqref{eq:P1_obj}-\eqref{eq:P1_comp_const}, \eqref{eq:P1_obj} represents the objective function, which for our case is to minimize line losses, i.e., $f(x)=\sum_{l=1}^L\sum_{t=t_0}^T\sum_{\phi=1}^{|\phi|}(\text{diag}(R_{l}\circ I_{l,t}))$. \added{The constraint that relates the voltages and currents in the network to the variables $W_{n,t}$, $I_{l,t}$ and $S_{l,t}$ are in~\eqref{eq:P1_BFM} while \eqref{eq:P1_volt_rel} is the power flow equation relating} the voltage drop in the network with the branch power flows. Constraint~\eqref{eq:P1_power_balance} represents the power balance equation at each node \added{which makes sure that the power coming into a node equals power going out, \eqref{eq:P1_node_real_balance} and \eqref{eq:P1_node_reactive_balance} are the real and reactive nodal power balance equations}, \eqref{eq:P1_line_const} is the line power flow constraint with $S_{\text{max},l} \in \mathbb{R}^{|\phi|}$ being the apparent power limit of line $l$, \eqref{eq:P1_volt_const} is the voltage limit constraint at each node with $V_{\text{min},n} \in \mathbb{R}^{|\phi|}$ and $V_{\text{max},n} \in \mathbb{R}^{|\phi|}$ the lower and upper voltage limit respectively at node $n$, and \eqref{eq:P1_solar_inv_limit} represents the apparent power limit of the solar inverter at node $n$. Constraints~\eqref{eq:P1_battery_inv_limit}-\eqref{eq:P1_comp_const} describe the battery power, state of charge (SoC) and charge/discharge complementarity constraints with $H_{\text{max},n} \in \mathbb{R}^{|\phi|}$ as the apparent power limit of the battery inverter at node $n$ and $B_{\text{min},n} \in \mathbb{R}^{|\phi|}$ and $B_{\text{max},n} \in \mathbb{R}^{|\phi|}$ as the lower and upper state of charge limit of the battery respectively at node $n$ and $\Delta t$ is the prediction horizon step. \added{The variable types used in the formulation are presented in Table~\ref{table_var}.}

\begin{table}
\centering\caption{\label{table_var} Variables used in the model formulation.}
\begin{tabular}{ll}
    \toprule 
        \textbf{Variable type}   & \textbf{Variables}\\
    \midrule
    \textbf{Decision}           & $P^{\text{d}}_{n,t}$, $P^{\text{c}}_{n,t}$, $q^{\text{b}}_{n,t}$, $S^{\text{S}}_{n,t}$  \\
    \textbf{Dependent}          & $W_{n,t}$, $S_{l,t}$, $I_{l,t}$, $S^{\text{net}}_{n,t}$, $B_{n,t}$ \\
    \textbf{Constant parameters} & $Z_l$, $S^{\text{L}}_{n,t}$, $S_{\text{max},l}$, $V_{\text{min},n}$, $V_{\text{max},n}$, $G_{\text{max},n}$, $\eta_{\text{c},n}$, \\
                & $\eta_{\text{d},n}$, $H_{\text{max},n}$, $\Delta t$, $B_{\text{min},n}$, $B_{\text{max},n}$, $P_{\text{max},n}$  \\
     \bottomrule
\end{tabular}
\end{table}

The optimization model from~\eqref{eq:P1_obj}-\eqref{eq:P1_comp_const} is nonlinear due to the equality constraints in~\eqref{eq:P1_BFM} and~\eqref{eq:P1_comp_const}, which can also be equivalently expressed as an integer constraint using binary variables as shown in~\cite{nazir2018receding}. These constraints make the problem NP-hard. The nonlinear equality constraint in~\eqref{eq:P1_BFM} can equivalently be expressed by the following two constraints~\cite{gan2014convex}:
 \begin{equation}\label{eq:SDP_relax}
\begin{bmatrix}
W_{n,t} & S_{l,t}\\
S_{l,t}^* & I_{l,t}\end{bmatrix} \succeq 0, \ \ \
rank\begin{bmatrix}W_{n,t} & S_{l,t}\\
S_{l,t}^* & I_{l,t}\end{bmatrix} = 1\\
\end{equation}
The inequality constraint in equation \eqref{eq:SDP_relax} is an positive semi-definte (PSD) convex constraint, whereas the rank constraint is non-convex. Removing the rank constraint in~\eqref{eq:SDP_relax} leads to a convex SDP formulation, however, it is desirable to find a second order cone relaxation that can be solved with numerically robust solvers such as GUROBI~\cite{gurobi}. SOCP relaxation can be applied to the PSD constraint in equation \eqref{eq:SDP_relax} as in~\cite{kocuk2016strong,horn2012matrix} to obtain the following relaxed SOC constraints:, 

\begin{align}\label{eq:SOC_relax1}
\left\lVert
\frac{2W_{n,t}(i,j)}
{W_{n,t}(i,i)-W_{n,t}(j,j)}\right\rVert
_2 & \leq W_{n,t}(i,i)+W_{n,t}(j,j)\\
\label{eq:SOC_relax2}
\left\lVert
\frac{2I_{l,t}(i,j)}
{I_{l,t}(i,i)-I_{l,t}(j,j)}\right\rVert
_2 & \leq I_{l,t}(i,i)+I_{l,t}(j,j)\\
\label{eq:SOC_relax3}
\left\lVert
\frac{2S_{l,t}(i,j)}
{W_{n,t}(i,i)-I_{l,t}(j,j)}\right\rVert
_2 & \leq W_{n,t}(i,i)+I_{l,t}(j,j)
\end{align}

If the complementarity constraint given in equation ~\eqref{eq:P1_comp_const} is also relaxed, the optimization model becomes convex and can be solved with GUROBI (as a QCQP) or MOSEK (as an SOCP). However, the reader will notice that removing ~\eqref{eq:P1_comp_const} means that a feasible solution may charge and discharge a battery simultaneously, which is not physically realizable. Therefore, we need to analyze conditions under which the complementarity condition is satisfied at optimality. Improving upon the work in~\cite{nazir2018receding}, Section \ref{comp_const} provides conditions for avoiding simultaneous charging and discharging in batteries which are not dependent 
on the size of inverters.

\section{Relaxing battery complementarity constraint}\label{comp_const}
This section focuses on the phenomenon of simultaneous charging and discharging (SCD) of batteries in~\eqref{eq:P1}. \added{As detailed in~\cite{nazir2018receding} and~\cite[Appendix]{almassalkhi2015model} and illustrated in Fig.~\ref{fig_battery_relaxation}, SCD begets a family of battery dispatch solutions $(P_{n,t}^\text{d}, P_{n,t}^\text{c})$  whose nodal net-injections, $P_{n,t}^\text{d} - P_{n,t}^\text{c}$, are identical but whose effect on the battery's predicted state of charge introduces a undesirable prediction error proportional to  $\sum_{t\in \mathcal{T}} P_{n,t}^\text{d} P_{n,t}^\text{c}$. 
To avoid SCD, one can enforce complementarity condition~\eqref{eq:P1_comp_const} between charging and discharging decision variables of each battery.}  However,~\eqref{eq:P1_comp_const} renders the SOCP problem non-convex. One approach to eliminate the challenging constraint is to introduce a binary (charge/discharge) variable to formulate an equivalent mixed-integer SOCP (MISOCP) problem. However, despite recent advances in MIP solvers, the MISOCP is computationally challenging as the number of batteries or the time-horizon increases. Instead, this paper omits~\eqref{eq:P1_comp_const} entirely and then analyzes under which conditions the optimal solution satisfies the complementarity constraint. This ensures that a (near) globally optimal solution can be achieved in a computationally efficient manner.

In~\cite{almassalkhi2015model} the authors present a receding-horizon OPF scheme for congestion management with batteries and quantify the effects of simultaneous charging and discharging. They then devise an algorithm that uses primal and dual variables from a relaxed solution to parameterize a  complementarity-enforced instance of the problem. That approach ensures the dispatch is physically realizable and does not modify the original objective function. However, the computational effort has now doubled by having to solve the problem twice) and it may still not be optimal. In~\cite{li2016sufficient}, Karush-Kuhn-Tucker (KKT) conditions applied to a balanced bulk transmission network for the economic dispatch problem are analyzed and it is shown that under reasonable economic assumptions, simultaneous charging and discharging can be avoided. \added{The authors in~\cite{li2016non} introduced a method of modifying the objective function in order to avoid simultaneous charging and discharging, but do not provide the conditions under which it holds}. In~\cite{nazir2018receding}, we provided a preliminary analysis in a loss-minimizing distribution system setting with energy storage by imposing sufficient conditions on the Lagrange multipliers from the energy balance and inverter constraints. However, the conditions proposed in~\cite{nazir2018receding} are impractical for systems with bidirectional loads (e.g., batteries) and a high penetration of solar PV. In this paper, more general conditions are proposed that provably guarantee no SCD and hold for different practical optimization objectives and use-cases. Specific operating conditions are finally identified where SCD is provably optimal (which is undesired) and explicit methods are then presented that enforce complementarity.

The approach herein first augments the objective function to reduce the effects of SCD's fictitious energy losses, i.e., fictitious in the sense that the predicted state of charge will be different from the actual state of charge since the battery cannot operate with SCD (e.g., see Fig.~\ref{fig_battery_relaxation}), and is as follows:
 \begin{equation}\label{eq:aug_obj}
 f(x)+\alpha \sum_{t=t_0}^T\sum_{n=1}^{|G|}\sum_{\phi=1}^{|\phi|} P^{\text{d}}_{n,t}\left(\frac{1}{\eta_{\text{d},n}}-\eta_{\text{c},n}\right),
 \end{equation}
  where $f(x)$ is given in~\eqref{eq:P1_obj}.
Loss-minimization on the IEEE-13 node network with network parameters provided in\cite{kersting2001radial}, Fig.~\ref{fig_simCandD} illustrates the effects of SCD with a single battery. Without the complementarity constraint imposed, the optimizer may waste energy through charge/discharge inefficiencies to achieve a lower state of charge of the battery as shown in Fig.~\ref{fig_SCD_SOC}

 \begin{figure}[!ht]
\centering
\includegraphics[width=0.48\textwidth]{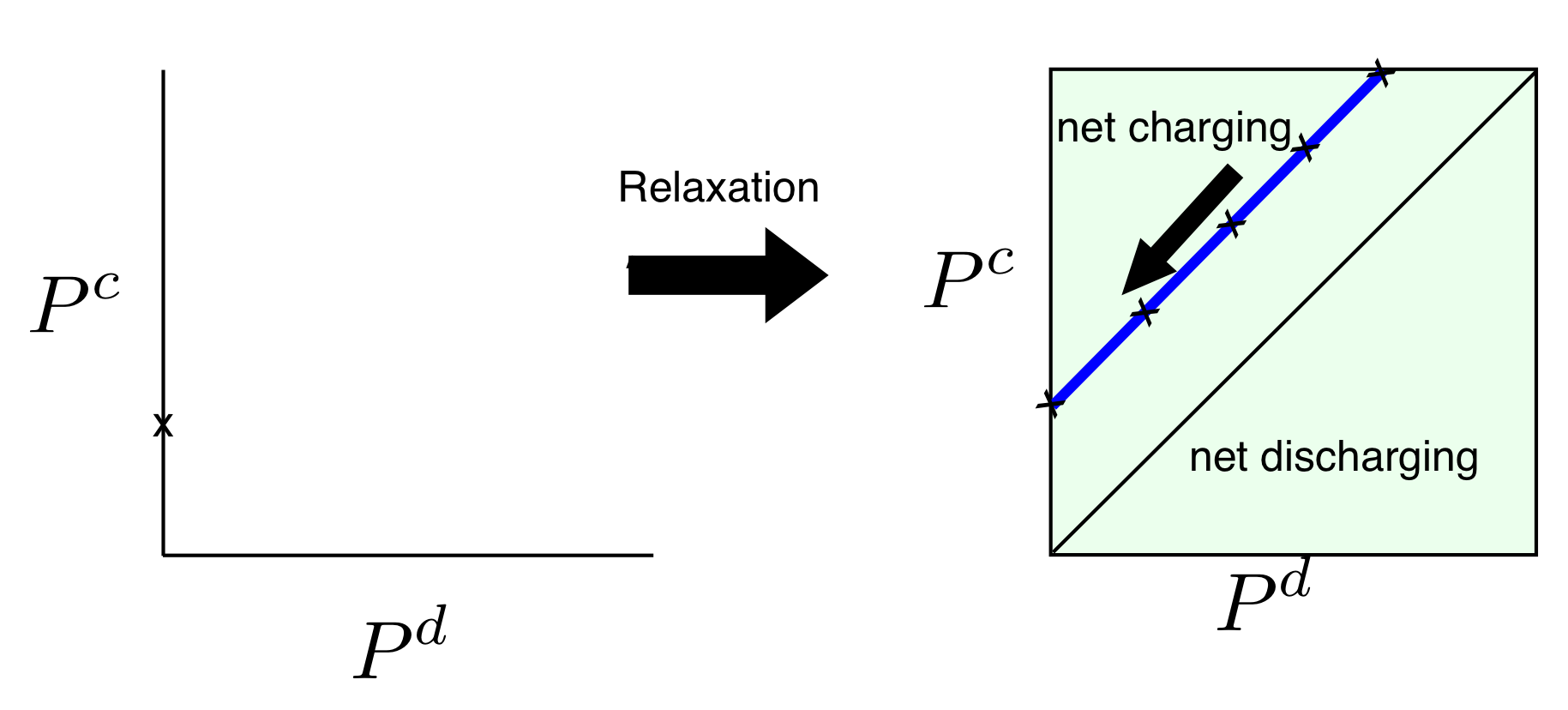}
\caption{Illustration of simultaneous charging and discharging (SCD) from relaxing the battery's complementarity constraint. (Left) SCD is enforced, so any net injection value, $P^{\text{d}}-P^{\text{c}}$, gives rise to only one solution.  (Right) the same net injection value gives rise to a family of solutions shown in blue where the battery's state of charge (SoC) are different due to SCD's so-called ``fictitious energy losses.''} 
\label{fig_battery_relaxation} 
\end{figure}


The addition of the battery power term in the objective function avoids SCD as shown in Fig.~\ref{fig_simCandD} with a negligible effect on the original objective function as illustrated with Fig.~\ref{fig_convex_MI}, where a comparison is presented with the exact mixed integer formulation. The solutions have the same optimized line losses as shown in Fig.~\ref{fig_convex_MI} and the addition of battery power term incentivizes the solution to points that satisfy the complementarity constraint.
\begin{figure}[!ht]
    \centering
    \includegraphics[width=0.41\textwidth]{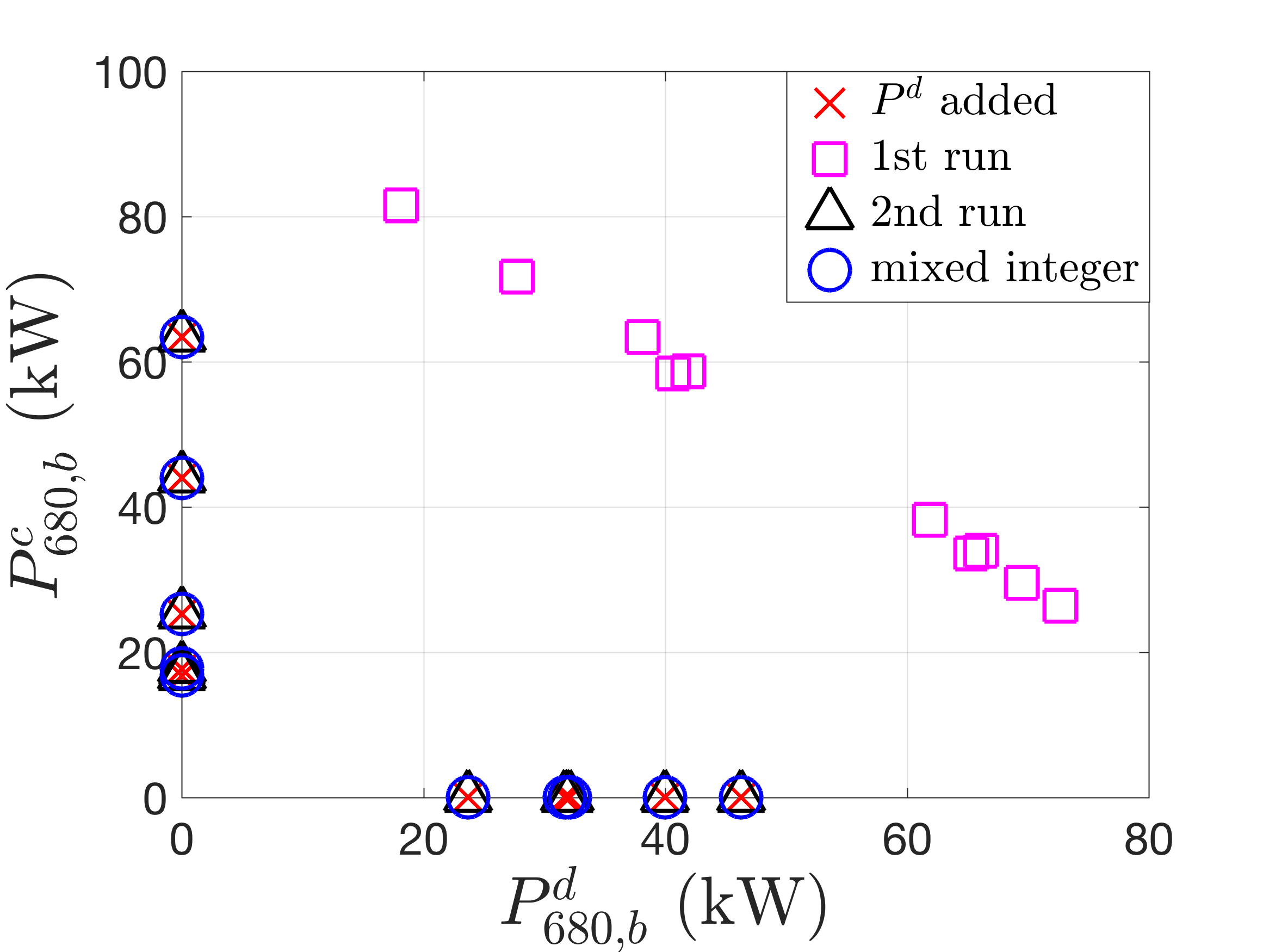}
  \caption{ Comparison of simultaneous charge and discharge in battery at node~680, phase $B$ for IEEE-13 node system between the cases with battery power term in objective, the 1st and 2nd run of the two-step algorithm presented in~\cite{almassalkhi2015model} and the mixed integer formulation. The reason for simultaneous occurrence of charge and discharge is that the objective function only has terms for the losses in the distribution lines and does not take into account the fictitious energy loss in the battery due to charging and discharging
  Thus, all solutions with the same value for $P^d-P^c$, are equivalent in the optimization solution, which begets simultaneous charging and discharging.}
   \label{fig_simCandD}
\end{figure}

\begin{figure} 
    \centering
  \subfloat [\label{fig_SCD_SOC}]{   \includegraphics[width=0.46\linewidth]{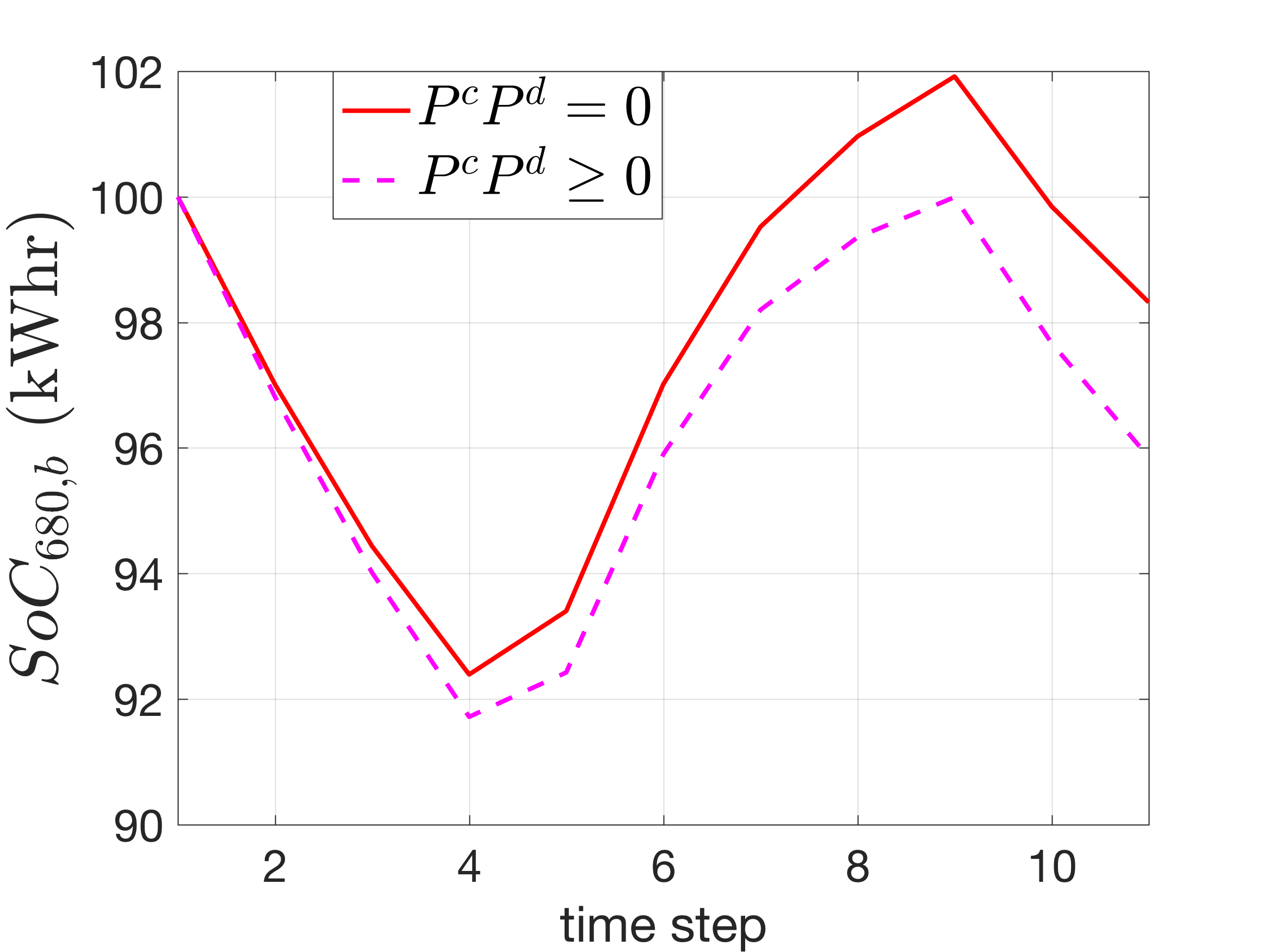}}
    \hfill
  \subfloat [\label{fig_convex_MI}]{    \includegraphics[width=0.48\linewidth]{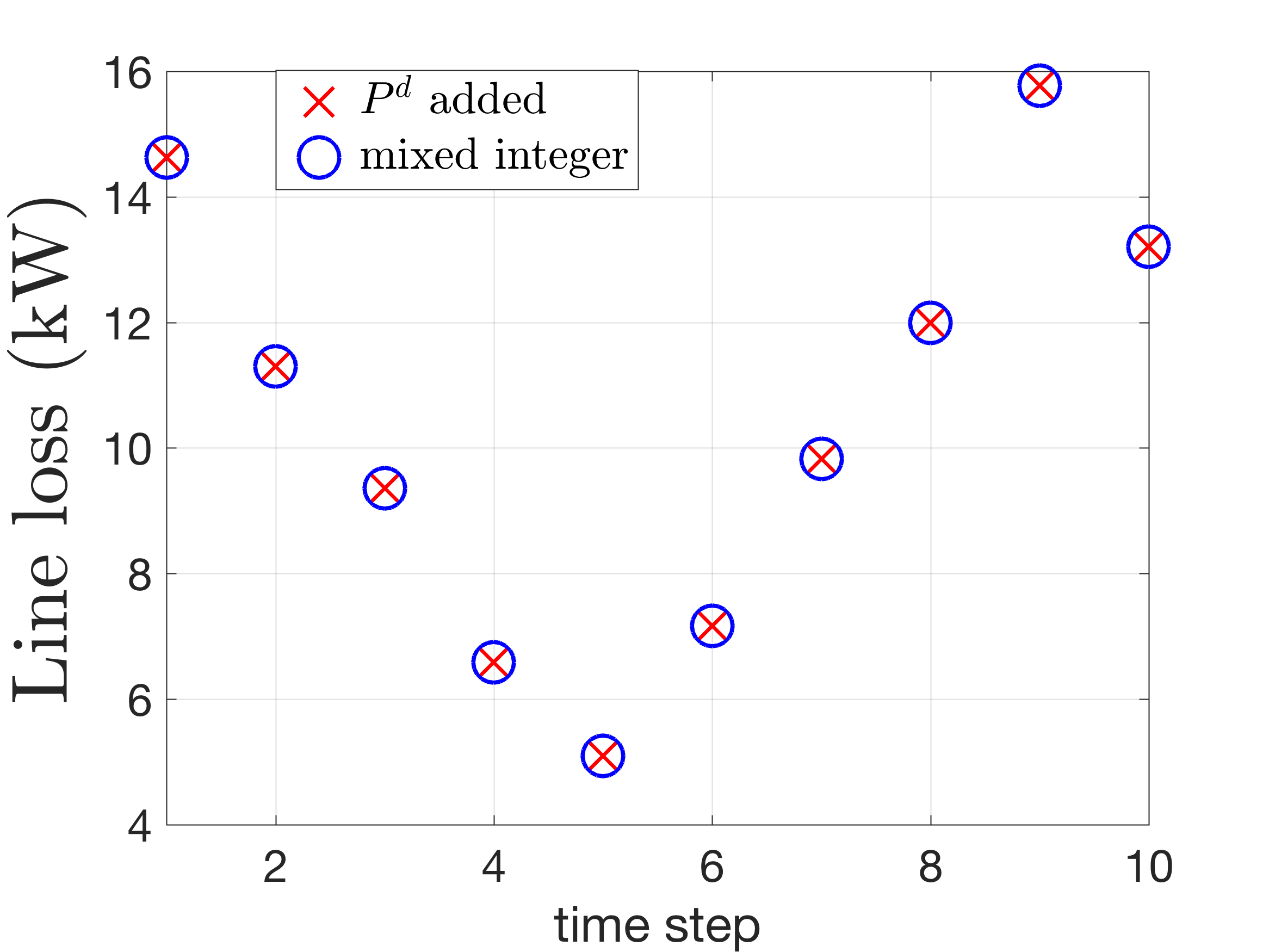}}
  \caption{(a) Comparison of state of charge of battery at node~680, phase $B$ for IEEE-13 node system with and without battery power term in objective. Due to the occurrence of simultaneous charge and discharge, energy is fictitiously consumed in the battery leading to a lower net state of charge. (b) Comparison of objective value (line loss) between the convex formulation and mixed integer formulation over a prediction horizon. The figure shows that the addition of battery power term to the objective of the convex formulation has negligible effect on the objective value of minimizing line losses.}
\end{figure}

To formalize this result, Theorem~\ref{Th1} below provides specific condition\added{s} under which the convex formulation with a differentiable objective function, $f(x)$ can avoid SCD with~\eqref{eq:aug_obj}. Specifically, the result holds for the following practical objectives in distribution networks:
\begin{itemize}
    \item $f(W)=\sum_i(W_i-W^{nom})^2$ (e.g., minimizing voltage deviation from nominal)\label{obj_1}
    \item $f(I)=\sum_l(\text{diag}(R_{l}\circ I_l))$ (e.g., minimizing network line loss)\label{obj_2}
    \item $f(P_{0,t})$=$(P_{0,t}-P^{\text{ref}})^2$ (i.e., tracking a grid/head-node reference power set-point)\label{obj_3}
    \item $f(P^{\text{d}}+P^{\text{c}})=\sum_i(P^d_i+P^c_i)$ (e.g., minimizing battery degradation)\label{obj_4}
    \item $f(P^{\text{d}}-P^{\text{c}})=(P^{\text{d}}-P^{\text{c}}-P^{\text{ref}})^2$ (i.e., tracking VB reference trajectory)
\end{itemize}

\begin{theorem}\label{Th1}
For the SOCP optimization problem~\eqref{eq:P1_obj}-\eqref{eq:P1_Pc_limit} with modified objective function~\eqref{eq:aug_obj}, the SCD relaxation is exact if the following conditions hold at each node $n$ and phase $\phi$:
\begin{enumerate}
    \item [C1:] $\frac{\partial f(x)}{\partial P^{\text{c}}}+\frac{\partial f(x)}{\partial P^{\text{d}}}\ge 0$,
    \item [C2:] $\alpha$ in~\eqref{eq:aug_obj} is strictly positive ($>0$),
    \item [C3:] \added{$\Gamma(t):=\sum_{\tau=t}^T(\beta_{1,n,\phi}(\tau)-\beta_{2,n,\phi}(\tau))\ge -\alpha$, $\beta_{1,n,\phi}(\tau), \beta_{2,n,\phi}(\tau) \in \mathbb{R}_+$ be Lagrange multipliers for the upper and lower bounds of inequality~\eqref{eq:P1_SOC_limit}, respectively.}
\end{enumerate}
\end{theorem}

\added{The proof of Theorem \ref{Th1} is provided in Appendix \ref{AppendixA}.}
 Theorem \ref{Th1} showed that with the modified objective function given by \eqref{eq:aug_obj}, SCD can be avoided under certain conditions in order to obtain a physically realizable solution from the optimizer. 

The addition of the battery power term in the objective does, however, modify the objective function  resulting in a sub-optimal solutions compared to the original objective. When the battery is charging, i.e.,  $P^{\text{d}}=0$, the modified objective is the same as the original objective resulting in the same optimal value. When the battery is discharging, i.e., $P^{\text{d}}>0$, the modified objective is different from the original line loss minimization objective, however, as $\alpha$ can be chosen to be small the effect on the optimal value is negligible.

Figure~\ref{fig_simCandD} shows the comparison between the solution ($P^{\text{d}}, P^{\text{c}}$) obtained from the convex formulation and the mixed-integer formulation and shows that the solutions match. This is also shown in Fig.~\ref{fig_convex_MI}  which compares the optimal value (line loss) between the two formulations and shows that the objective values match. 

\begin{remark}
Theorem \ref{Th1} holds for the given objectives when conditions C1,~C2 and~C3 are satisfied. However, condition C3 can be restrictive, especially under a high penetration of renewable generation. Furthermore, certain objective functions like tracking battery state of charge require stricter conditions as shown in Corollary~\ref{corollary} in Appendix~\ref{AppendixB}. For such cases, the following methods are proposed to obtain a physically realizable solution that avoids simultaneous charging and discharging of batteries. 
\begin{itemize}
    \item  Large $\alpha$: In the case where $\Gamma(t)<0$, $\alpha$ can be chosen large enough to ensure that condition C1 is satisfied. The value of $\Gamma(t)$ may be estimated based upon the solar and load conditions. However, the drawback of this approach is that a large value of $\alpha$ would clearly shift the optimal solution.
    
    \item  Two-step battery dispatch: as presented in~\cite{almassalkhi2015model}, the first run permits SCD and the the second enforces complimentarity based on the net-effect of the first solution to obtain a physically realizable solution. This method can provide a near optimal feasible solution as shown in Fig.~\ref{fig_simCandD} but doubles the run-time. Future work will further explore how this two-step technique can provide optimality certificates and will improve its implementation to avoid the doubling of the solve time.
    
    \item  Simplified battery model: In this method, the battery model in \eqref{eq:P1_battery_power_rel} is replaced by an approximated battery model that uses a single standing-loss efficiency $\eta_{eq}$ instead as shown below:
    \begin{align*}
        B_{n,t+1}=\eta_{eq}B_{n,t} - \Delta t P^{\text{b}}_{n,t}
    \end{align*}
    where $P^{\text{b}}_{n,t}\in \mathbf{R}^{|\phi|}$ is the net battery power injection. The value of $\eta_{eq}$ can be estimated based upon expected battery schedule and the values of $\eta_d$ and $\eta_c$. Future work will explore a mapping between the two battery models as a way to estimate $\eta_{eq}$ to minimize modeling error over the horizon with respect to the actual battery model in~\eqref{eq:P1_battery_power_rel}.
\end{itemize}

 
 \end{remark}
 
Based on the results of this section, a convex formulation of the multi-period three-phase OPF can be obtained that satisfies the complementarity condition between charging and discharging of batteries under certain conditions. When the conditions are not satisfied, this work proposes techniques to obtain a near optimal solution that enforces complementarity. However, the second order cone relaxation of the nonlinear power flow equations may engender solutions that are not physically realizable. To guarantee realizability, the next section presents a nonlinear programming (NLP) formulation of the OPF problem that is initialized with the relaxed SOCP solution. Note that the NLP initialization goes beyond just a warm-start and includes a novel mechanism to account for the multi-period formulation inherent to an energy storage trajectory. 

\section{Multi-period coupling of SOCP with NLP}\label{SOCP_NLP}

The original OPF formulation given by~ \eqref{eq:P1_obj}-\eqref{eq:P1_comp_const} is non-convex because of the nonlinear power flow constraint in~\eqref{eq:P1_BFM} and the SCD complementarity constraint in~\eqref{eq:P1_comp_const}. The two constraints are relaxed  to obtain an SOCP formulation of the OPF problem. The non-convex constraint~\eqref{eq:P1_comp_const} is relaxed as explained in Section~\ref{comp_const}, which provides conditions under which the SOCP solution is tight (with respect to the complementarity condition). 
However, the relaxation of the nonlinear power flow model in~\eqref{eq:P1_BFM}  with the second-order cone constraints~\eqref{eq:SOC_relax1}-\eqref{eq:SOC_relax3} can result in non-physical solutions to the OPF problem as has recently been shown in~\cite{wang2018chordal}.

Thus, if we seek a physically realizable solution for a general objective function, we need a nonlinear programming (NLP) formulation.  Thus, we seek to leverage the multi-period solution available from the SOCP. However, NLPs are not scalable and the solve time increases dramatically with the increase in problem size (and coupling)~\cite{boyd2004convex}
To overcome this challenge, we propose a time-decoupled approach by fixing the battery's active power set-points in the NLP based on the solution obtained from the SOCP. This allows the NLP to focus on reactive power set-points and voltage limits, which aligns with recent analysis~\cite{molzahn2017computing}. In~\cite{molzahn2017computing}, it is shown how reactive power and voltage limits lead to disconnections in the power flow solution space resulting in a non-zero duality gap for the relaxed OPF. Based on these observations, an SOCP-NLP coupled algorithm is developed  as shown in Fig.~\ref{fig:SOCP_NLP}, where the solution obtained from the SOCP is passed to the NLP solver. Prior work in literature, such as~\cite{kocuk2016strong}, have proposed the idea of using the solution of a convex relaxation as an initial starting point for solving the full, nonlinear ACOPF. However, herein we extend the notion of a ``warm start'' to the multi-period domain. Specifically, we decouple the multi-period NLP by fixing the active power set-points of the batteries to the solution of the convex relaxation (SOCP).
Keeping the active power solutions constant leads to fixing the state of charge of the batteries and, as a result, results in a decoupling of the time-steps of the prediction horizon in the NLP. Thus, each time-step can be solved independently and in parallel (as independent NLPs), which  leads to a scalable implementation compared to solving the multi-period NLP.

\begin{figure}[h]
\centering
\includegraphics[width=0.5\textwidth]{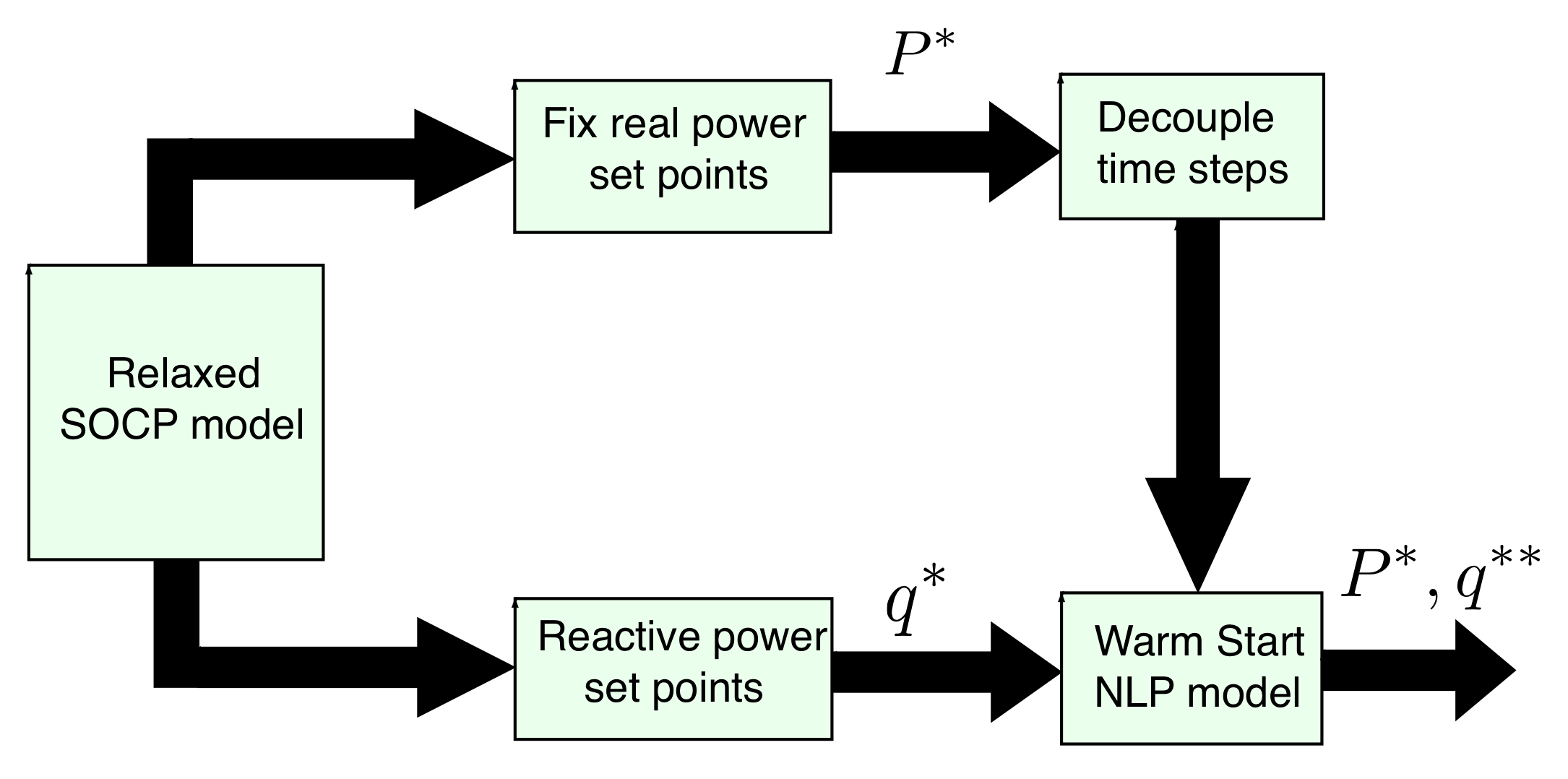}
\caption{\label{fig:SOCP_NLP}Coupling of SOCP with NLP by fixing real power solutions from SOCP and hence decoupling the NLP to obtain a feasible solution.}
\end{figure}

This is further explained through Fig.~\ref{fig:SOCP_NLP_time} where for each time-step of the prediction, the reactive power range available to the NLP is constrained by the SOCP's solution. \added{That is, Figure~\ref{fig:SOCP_NLP_time} illustrates the decomposition approach presented herein by showing the effect of the SOCP's optimized active power trajectory on the feasible set of the reactive power of the NLP problem.}
\begin{figure}[h]
\centering
\includegraphics[width=0.39\textwidth]{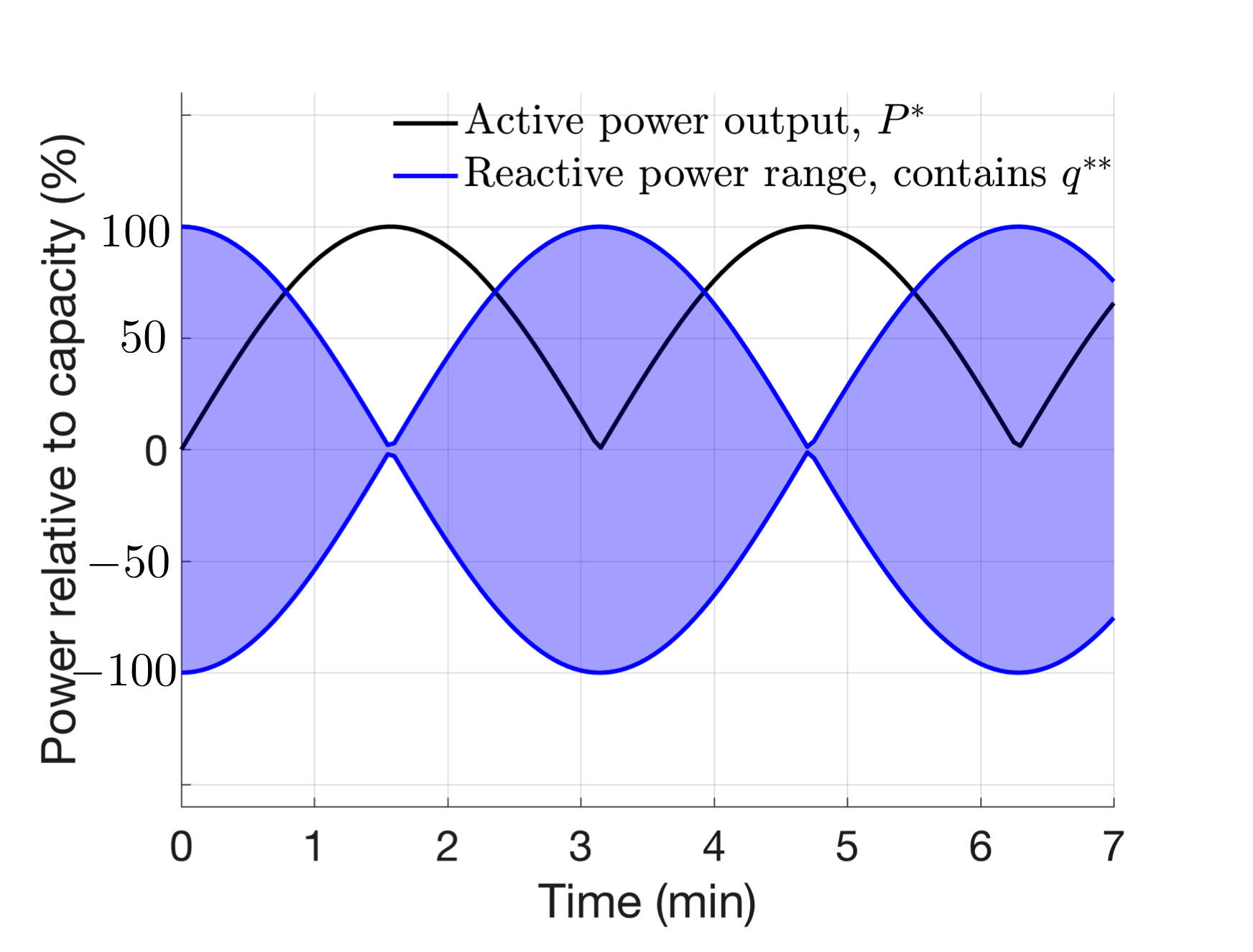}
\caption{\label{fig:SOCP_NLP_time} Available reactive power variation range for NLP across multiple time steps based on the active power trajectory provided by the SOCP.}
\end{figure}
\begin{remark}

The decoupling of the time-steps reduces the feasible set of the optimization problem and, hence, increases the optimal value. Thus, the decoupled problem represents an upper bound on the time-coupled nonlinear problem, which in turn is lower bounded by the SOCP as shown below:
\begin{equation}\label{soc_nlp_relation}
    SOCP_{opt}\leq NLP_{opt}\leq DNLP_{opt}
\end{equation}
where $NLP_{opt}$ represents the optimal value of the time-coupled nonlinear program and $DNLP_{opt}$ the optimal value of the time-decoupled nonlinear program (DNLP).
\end{remark}
The DNLP problem at each time-step $t$ of the prediction horizon can then be expressed as:
\begin{subequations}\label{P2}
\begin{equation}\label{eq:P2_obj}
\min_x\ \sum_{l=1}^L\sum_{\phi=1}^{|\phi|}\text{diag}(R_{l}\circ I_{l,t})
\end{equation}
\begin{equation}\label{eq:P2_ref}
  s.t:\  \eqref{eq:P1_BFM}-\eqref{eq:P1_battery_inv_limit}
\end{equation}
\begin{equation}\label{eq:P2_Pdfix}
    P^{\text{d}}_{n,t}=P^{\text{d}*}
\end{equation}
\begin{equation}\label{eq:P2_Pcfix}
    P^{\text{c}}_{n,t}=P^{\text{c}*}
\end{equation}
\end{subequations}
where $P^{\text{c}*} \in \mathcal{R}^{|\phi|}$ and $P^{\text{d}*} \in \mathcal{R}^{|\phi|}$ are the charge and discharge power of the battery obtained from the SOCP at node $n$ and time $t$, such that $P^*=P^{\text{d}*}-P^{\text{c}*}$. The NLP given by equations \eqref{eq:P2_obj}-\eqref{eq:P2_Pcfix} is solved separately at each step of the prediction horizon to obtain a feasible plus (near) optimal solution with guaranteed feasibility and a bound on the optimality, as the relaxed SOCP provides a lower bound on the optimal value of the original nonlinear problem~\cite{boyd2004convex}. Utilizing this SOCP-NLP coupled optimization framework, a scalable solution of three-phase OPF problem can be obtained rapidly, plus the framework provides bounds and guarantees on feasibility and optimality of the solution, where the upper-bound on the global optimality gap is computed from
\begin{equation}\label{eq:opt_gap}
    \text{\% optimality gap} \le \frac{DNLP_{opt}-SOCP_{opt}}{DNLP_{opt}}\times 100.
\end{equation}
In the next section, simulation tests are conducted on unbalanced IEEE test feeders to verify the feasibility of the proposed formulation and investigate the global optimality gap. The validation is conducted by
using forward-backward sweep in GridLab-D.

\section{Test case results and validation}\label{sim_results}
\subsection{\added{Case study description}}
Simulation-based analysis of the multiperiod SOCP-NLP algorithm presented above is conducted on the unbalanced 123-node IEEE test feeder with a base voltage of 2.4kV and base apparent power of 1 MVA. The algorithm is implemented in receding-horizon fashion. That is, the SOCP results in an open-loop, optimal battery and inverter control schedule, which is used by the NLP  to calculate a physically realizable schedule that is implemented by GridLab-D (i.e., the ``plant'') to determine the resulting AC power flows. The forecasts of demand and renewable generation are then updated and the SOCP-NLP implementation repeats. A sample forecast of aggregate solar, demand and net-demand over a prediction horizon is shown in Fig. ~\ref{fig:net_demand_profile}.
\begin{figure}[h]
\centering
\includegraphics[width=0.37\textwidth]{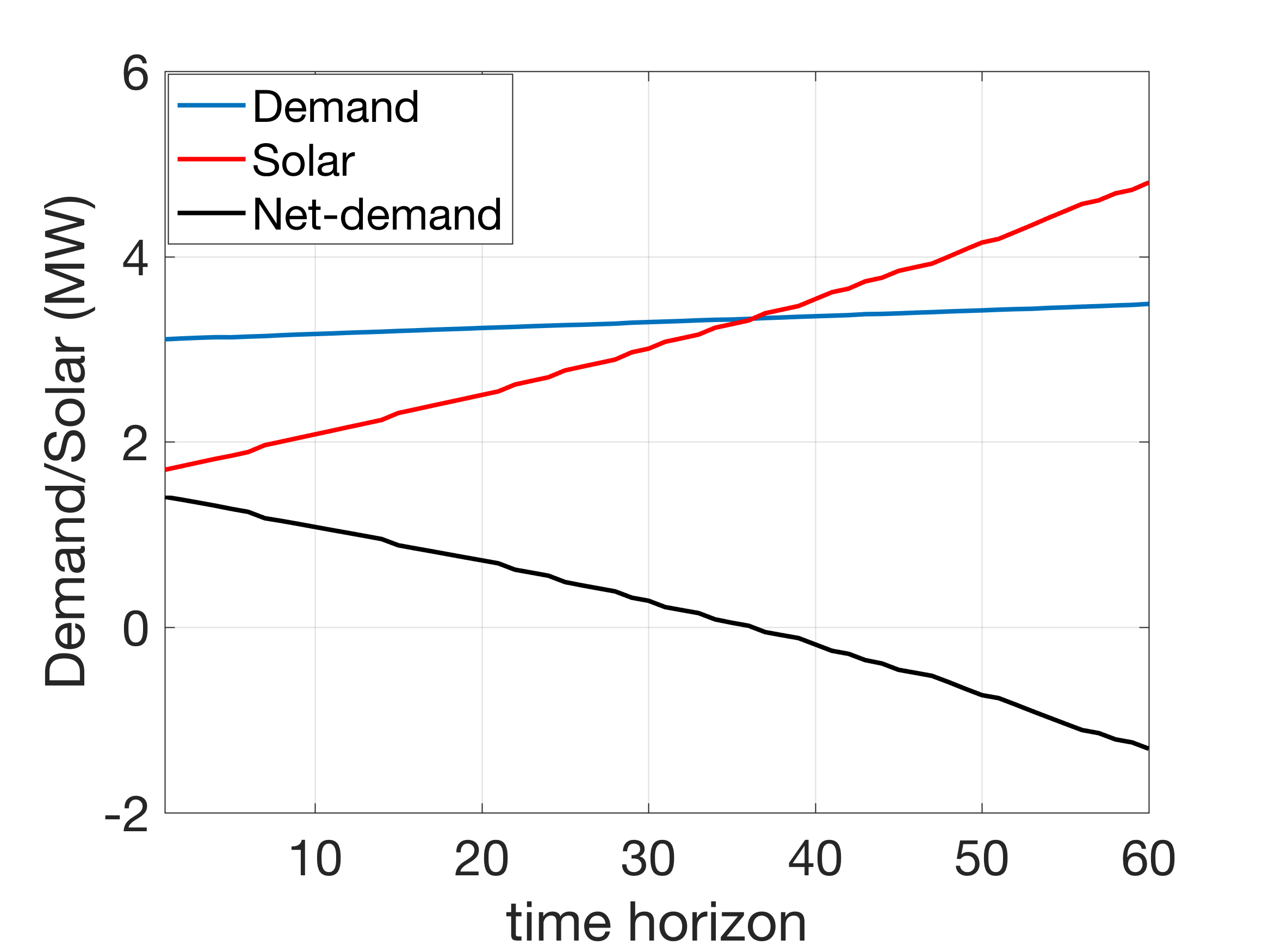}
\caption{\label{fig:net_demand_profile} Aggregate solar, demand and net-demand profile over a prediction horizon.}
\end{figure}

Distributed storage and solar PV units are added at random to 16 nodes in the network which can supply active and reactive power through four quadrant operation. Each storage unit has an energy capacity of 40 kWh and an apparent power rating of 50 kVA, whereas each solar PV unit has a rating of 100 kVA. The solar and load profile over the 30-step prediction horizon, with each step being 1~minute duration, are constructed from the minutely forecast data available~\cite{bing2012solar}. 
Let mean load, $\mu_l$, be the base load of the IEEE-123 node system and $\mu_s$ be the mean solar and equal to 100kW 
Discrete control devices such as switches, capacitor banks and transformers are fixed at their nominal value for this study. A three phase OPF is run in a receding horizon fashion with a prediction horizon of 30 time-steps, for the dispatch of controllable assets of the network to minimize the network losses. The set-points provided by the solution of the SOCP are used to initialize an NLP to provide a feasible solution. The SOCP is modeled in JuMP~\cite{dunning2017jump}, with Julia and solved using GUROBI~\cite{gurobi}. \added{The multi-period SOCP has 108,000 variables, 48,000 linear constraints and 81,000 SOC constraints.} The NLP is also modeled in JuMP, but solved with IPOPT~\cite{wachter2006implementation} using {\tt \small HSL\_MA86} solver~\cite{hsl2007collection}. \added{The single-period NLP has 3,600 variables, 1,600 linear constraints, 700 SOC constraints and 2,000 non-linear constraints.}
\subsection{\added{Data management}}
\added{To enable the presented framework, it is assumed that the minutely PV production forecast data and the demand profile data over the 30 minute horizon, as shown in Fig.~\ref{fig:net_demand_profile}, is available to the central dispatcher. In this paper, we assume a perfect forecast, whereas future work will investigate the role of uncertainty and robustness to forecast error. Such minutely solar PV forecasts  are  available today  at  minutely  forecasts with a 60-minute prediction horizon~\cite{bing2012solar}. It is also assumed that the dispatcher knows about the power rating and capacity of the available PV units and the updated state of charge of the distributed storage units. This dispatcher could be a distribution system operator (DSO; e.g., NY REV's DSIP~\cite{conED}), so it is reasonable to assume that such system information is available. Furthermore, it is assumed that the DSO knows the network topology, so it can formulate and solve the optimization problem based on network parameters and dispatch available flexible resources accordingly.}

\subsection{\added{Test case results}}
The results obtained under four different solar and load cases as shown in Table~\ref{table_cases}, where high load and high solar corresponds to the base values and low load and low solar corresponds to a mean value of 50\% of base. These cases are utilized to show the feasibility, optimality, gap and solve time of the formulated algorithms.
\begin{table}
\centering\caption{\label{table_cases} Different solar and load cases.}
\begin{tabular}{rcc}
    \toprule 
                            & Low load: 50\% load & High load: 100\% load \\
    \midrule
    Low solar: 50\% Solar   & Case LL & Case HL  \\
    High solar: 100\% Solar & Case LH & Case HH  \\
     \bottomrule
\end{tabular}
\end{table}

 The result in Table~\ref{table_optvalue} show that the optimality gap of the obtained solution is always less than 3\%, as the SOCP solution provides the lower bound to the global optimum. \added{The RMSE and worst case values are calculated based on the optimal values of the SOCP and the NLP run in a receding horizon fashion through simulations over a horizon of one hour.}
 \added{To further investigate the optimality gap, Fig.~\ref{fig:opt_gap_sweep} shows the optimality gap as the solar penetration level in the system is varied for the base load case. From the figure it can be seen that the optimality gap is always below 3\% for the different solar penetration levels.}
 
 The feasiblity of the NLP solution is tested against GridLab-D and the validation is given in Fig. \ref{fig:feasibility_NLP} for case HH, which shows that the voltages obtained from the NLP match closely with those obtained through a power flow performed in GridLab-D using backward-forward sweep. \added{NLPs are not scalable and the solve time increases dramatically with the increase in problem size (and coupling) as can be seen from Fig.~\ref{fig:time_horizon_nlp} for case HH, which shows the increase in solve time as the length of the prediction horizon increases}. The computation time of the \added{decoupled} algorithm is illustrated in Table \ref{table_solvetime} showing that the mean total solve time at each time step for SOCP+NLP is always under 45 seconds, providing sufficient time for communication delays in order to guarantee a solve time of under one minute for the dispatch of distributed storage to counter the fast-time variation in renewable generation. \added{The SOCP time is the time it takes to solve the multi-period optimization with a time-horizon of 30 steps, whereas the NLP time is the time it takes to solve each time-instant in a decoupled and parallel form.} Figures \ref{fig:reactive_profile1}-\ref{fig:reactive_profile4} shows the worst case difference in DER reactive power generation over the prediction horizon between the SOCP and the NLP, whereas Fig. \ref{fig:time_horizon} shows the variation in SOCP solve time as the size of the receding horizon is increased for case HH. \added{In the figure, the edges of the box represent the 25th and 75th percentile of data, whereas the $+$ sign represents the most extreme value in the dataset.} It can be seen from Fig. \ref{fig:time_horizon} that the SOCP algorithm scales well with the increase in horizon size. 



\begin{table}[h!]
\centering
\caption{\label{table_optvalue}Comparison of optimality gap values of SOCP and NLP.}
{
\begin{tabular}{r c c}
\toprule
Case & RMSE & Worst case \\
\midrule
            LL & 0.43 & 1.25\\
            HL & 0.83 & 0.91\\
            LH & 1.14 & 1.33\\
            HH & 0.88 & 2.10
\\ \bottomrule
\end{tabular}
}
\end{table}


\begin{figure}[h]
\centering
\subfloat[\label{fig:opt_gap_sweep}]{
\includegraphics[width=0.49\linewidth]{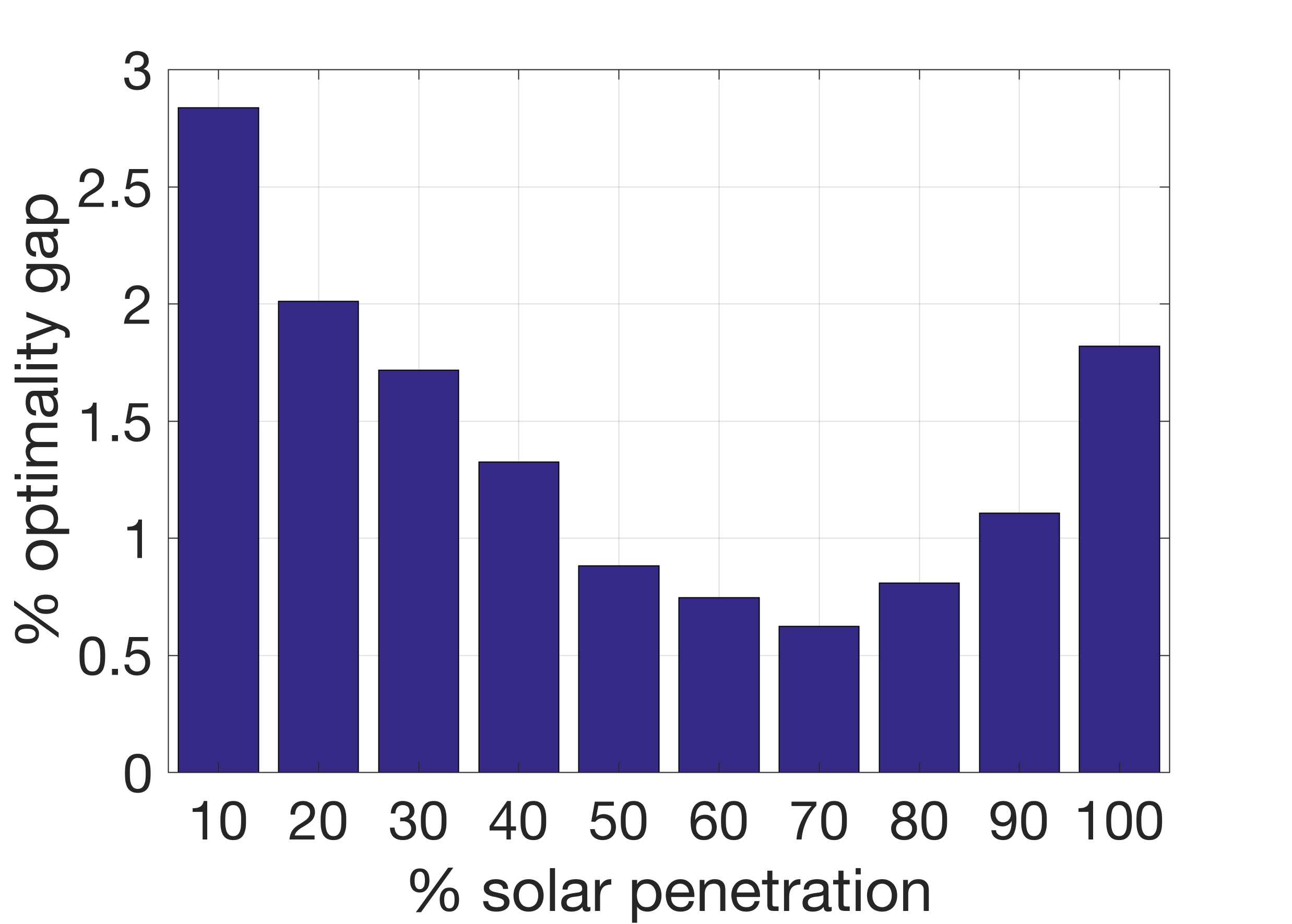}}
    \hfill
\subfloat[\label{fig:feasibility_NLP}]{
\includegraphics[width=0.49\linewidth]{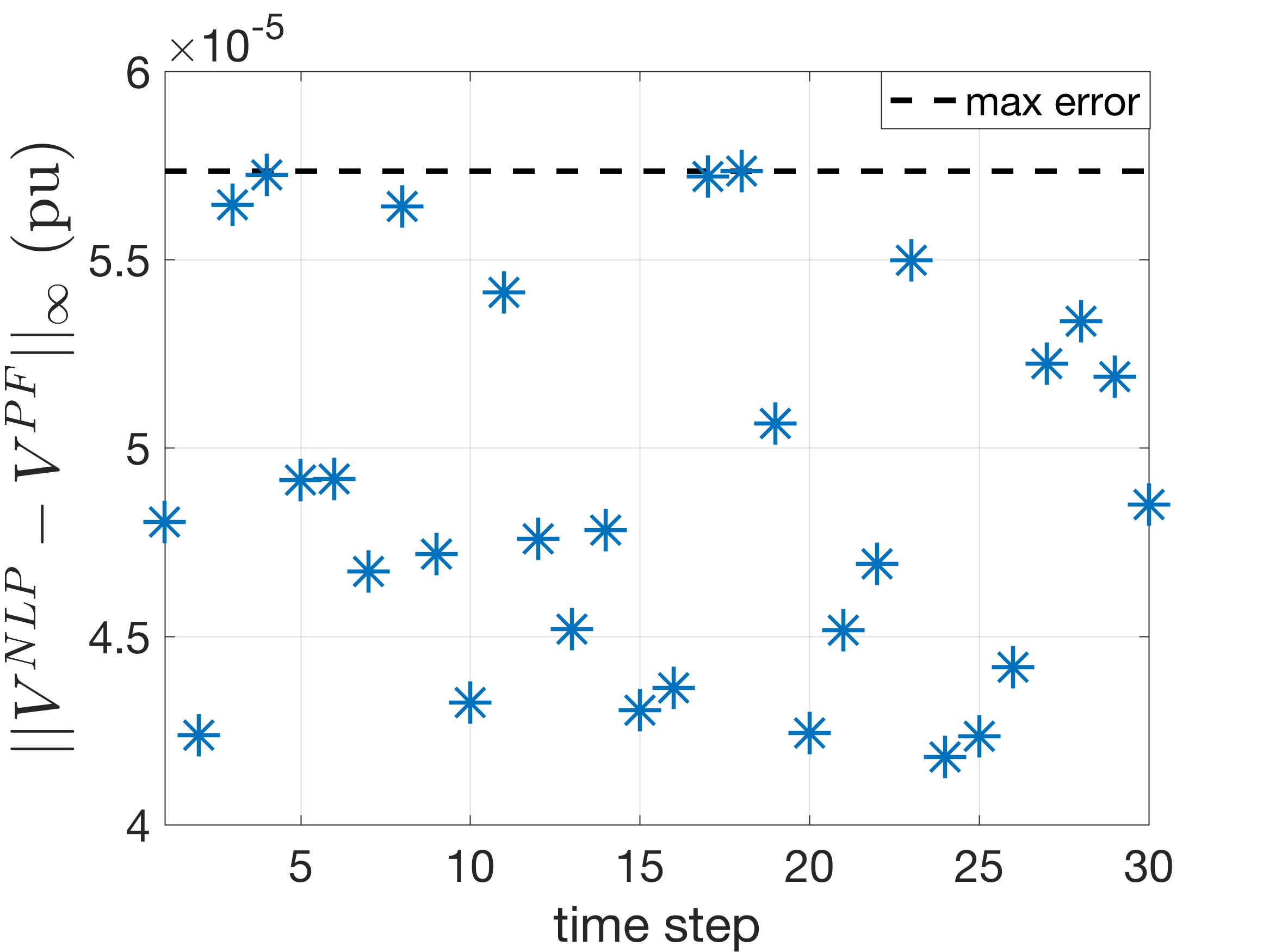}}
\caption{(a) Variation of optimality gap with change in \% solar penetration. (b)Worst case voltage error between NLP and power flow (PF) in GridLab-D over the time horizon showing the feasibility of the NLP solution.}
\end{figure}

\begin{figure} \label{fig:reactive_profile}
    \centering
  \subfloat[\label{fig:reactive_profile1}]{%
       \includegraphics[width=0.48\linewidth]{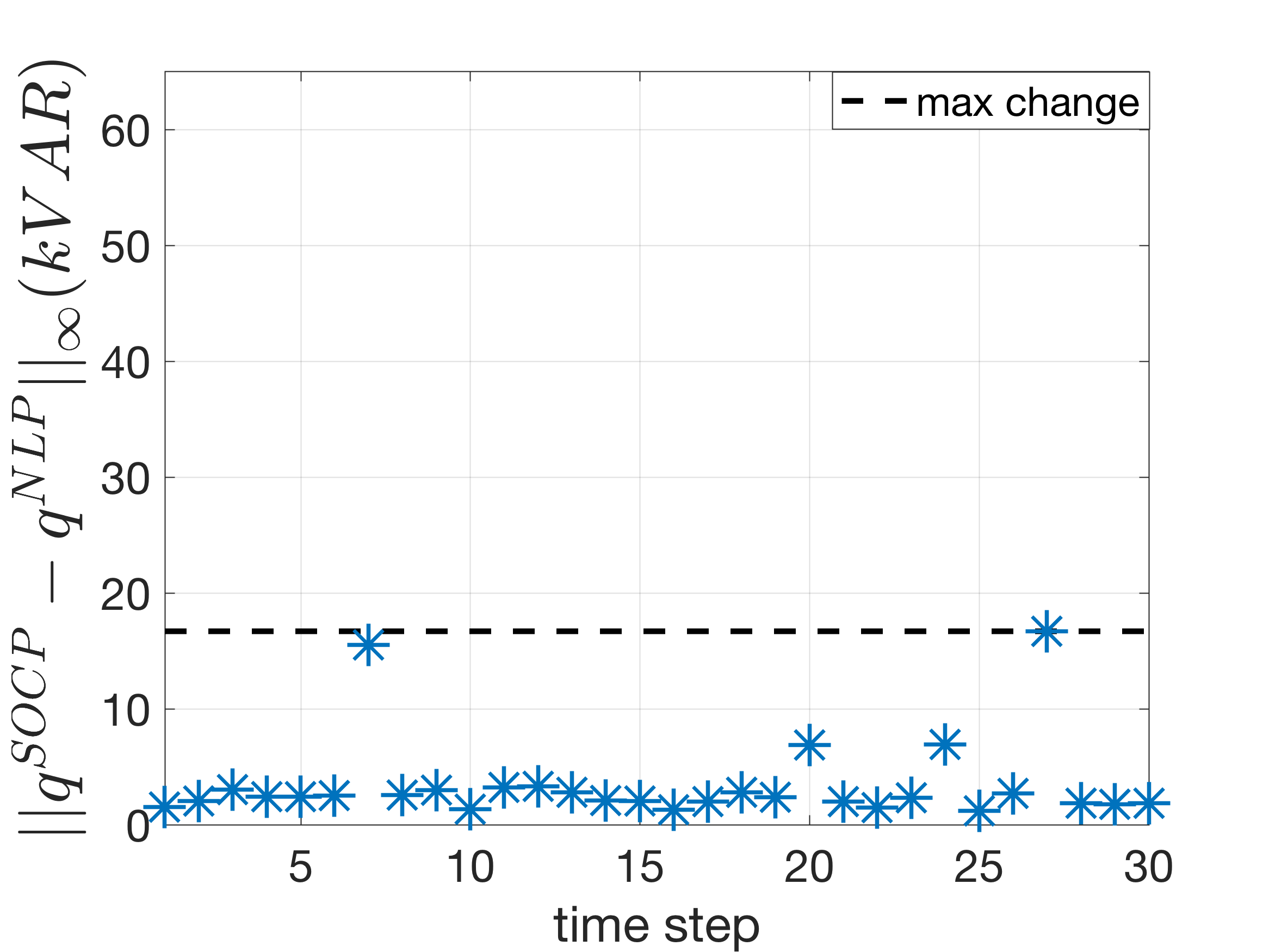}}
    \hfill
  \subfloat[\label{fig:reactive_profile2}]{%
        \includegraphics[width=0.48\linewidth]{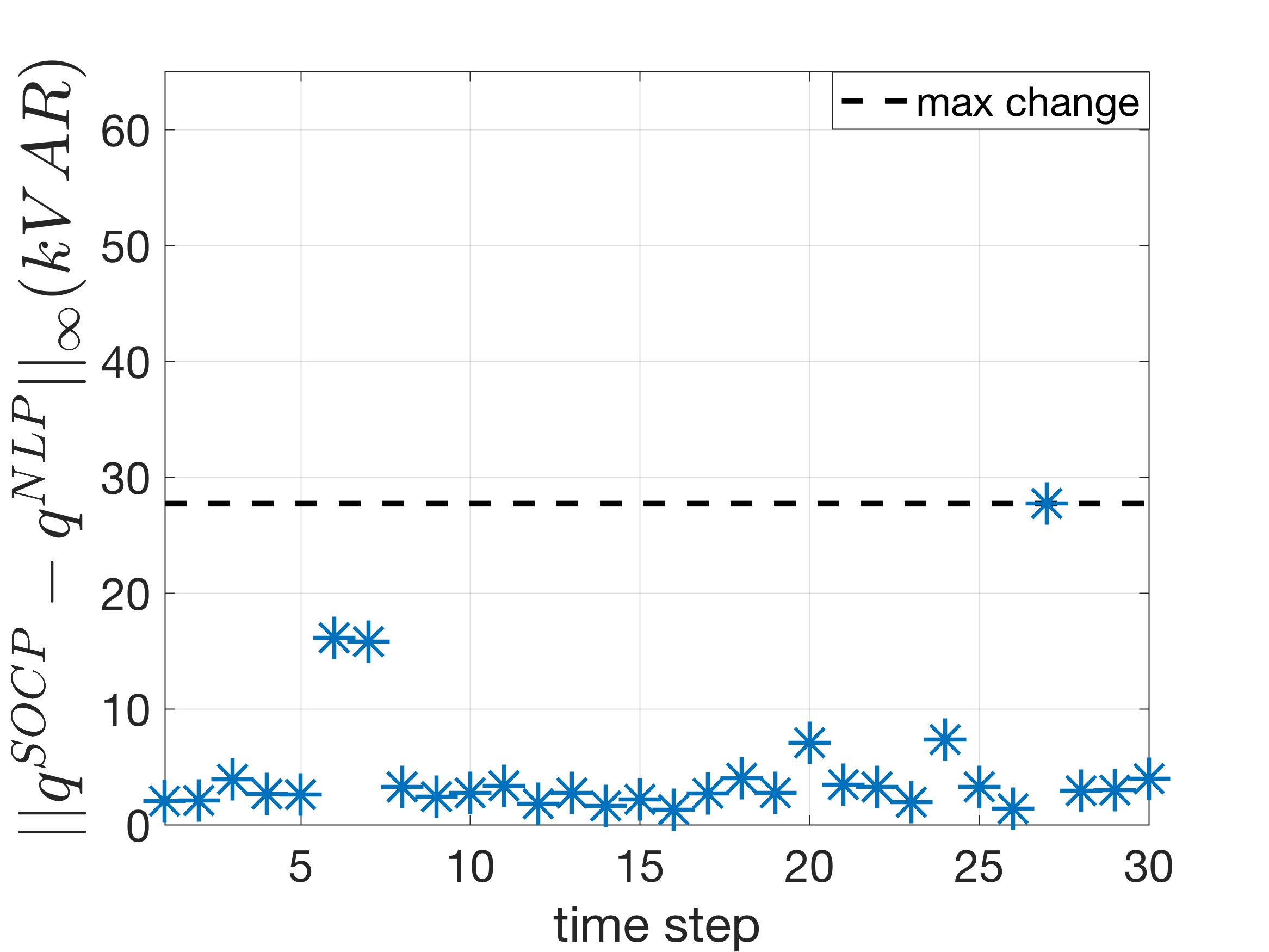}}
    \\
  \subfloat[\label{fig:reactive_profile3}]{%
            \includegraphics[width=0.48\linewidth]{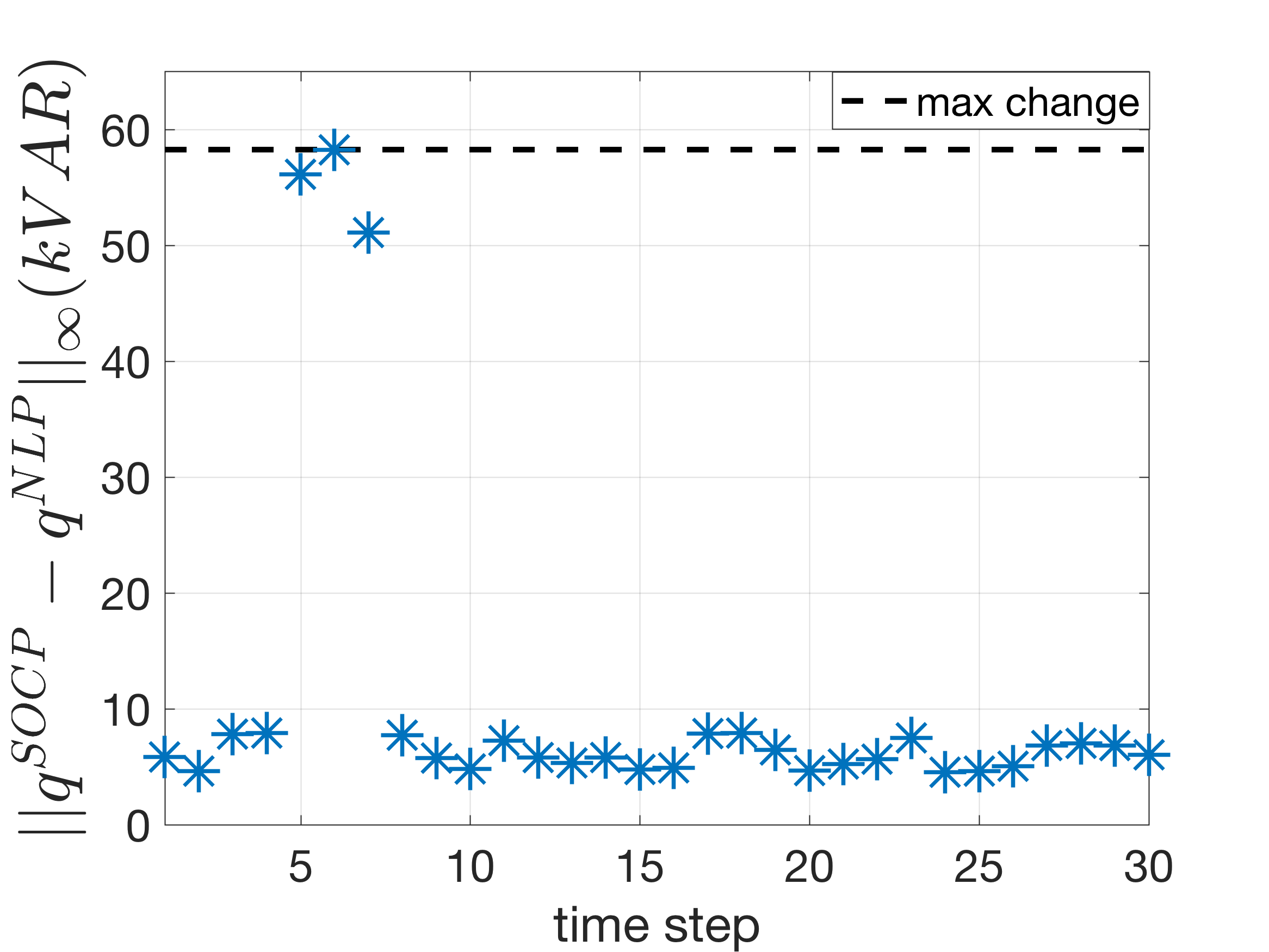}}
    \hfill
  \subfloat[\label{fig:reactive_profile4}]{%
        \includegraphics[width=0.48\linewidth]{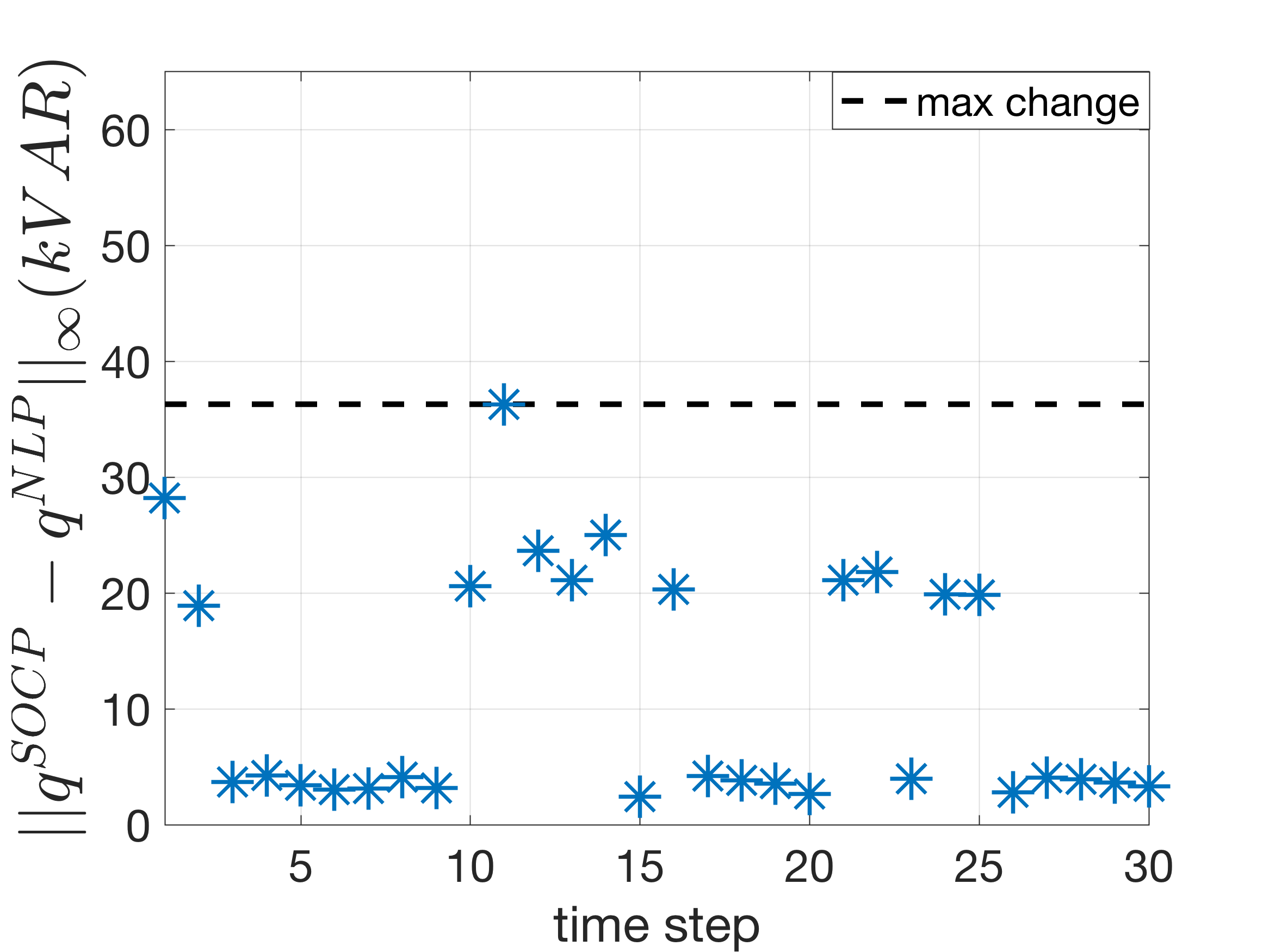}}
  \caption{Comparison of reactive power generation obtained from NLP and SOCP for IEEE-123 node system under the following cases: (a) low load, low solar (b) high load, low solar (c) low load, high solar (d) high load, high solar.}
\end{figure}


\begin{table}[h!]
\centering
\caption{\label{table_solvetime}Comparing solver times for the SOCP-NLP algorithm.}
{
\begin{tabular}{l c c c c}
\toprule
Solver time (s) & Case LL & Case HL & Case LH & Case HH \\
\midrule
            $(\mu,\, \sigma)_\text{SOCP}$ & (42.4, 6.3) & (17.9, 1.5) & (31.4, 9.6) & (24.1, 2.2)\\
            $(\mu_,\, \sigma)_\text{NLP}$ & (1.8, 0.3) & (2.2, 0.8) & (1.9, 0.2) & (2.1, 0.7)\\
            $(\mu,\, \sigma)_\text{total}$ & (44.2, 6.4) & (20.1, 1.6) & (33.3, 9.6) & (26.3, 2.5)\\
\bottomrule
\end{tabular}
}
\end{table}


\begin{figure}[h]
\centering
\subfloat[\label{fig:time_horizon_nlp}]{
\includegraphics[width=0.49\linewidth]{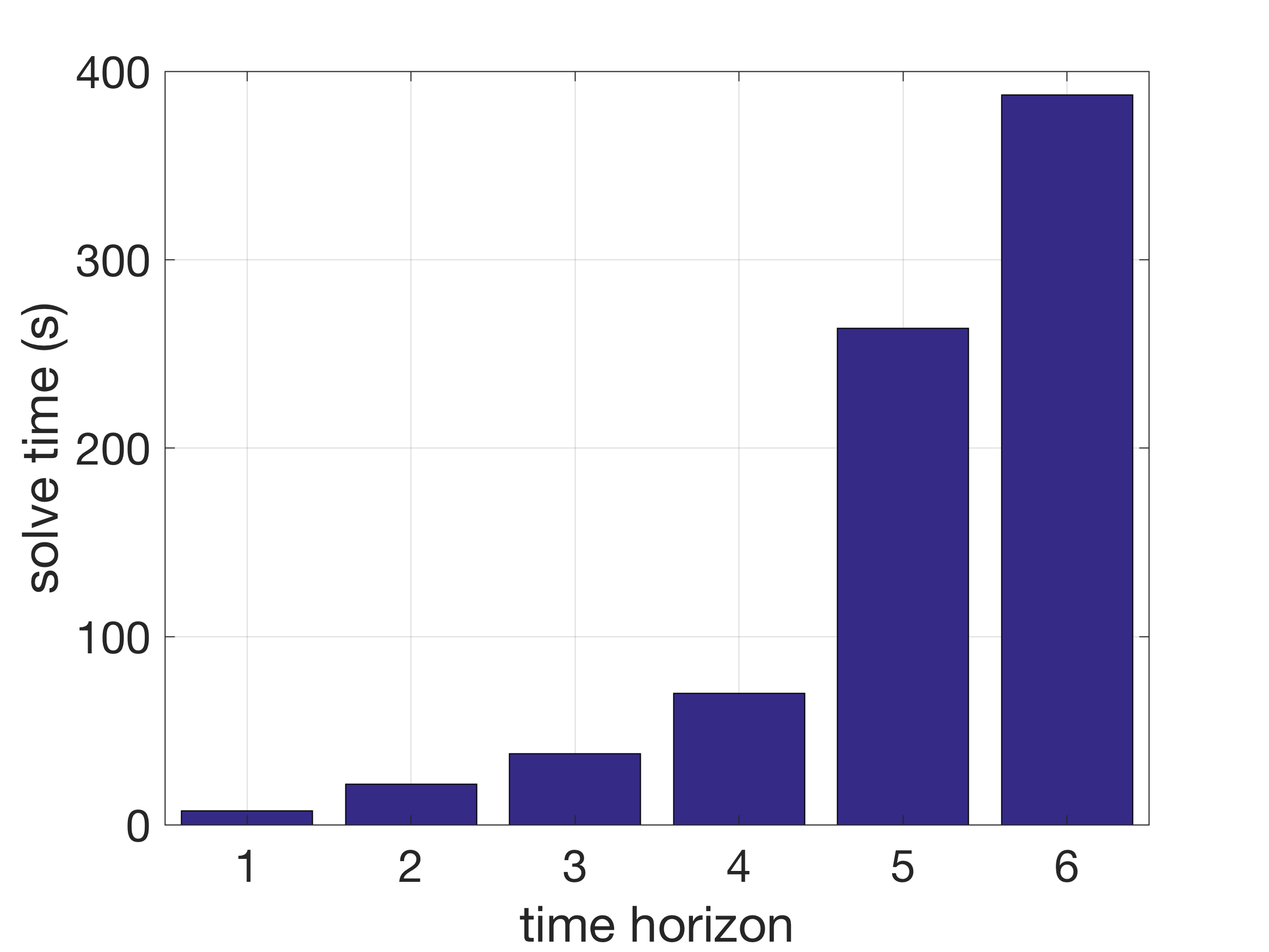}}
    \hfill
\subfloat[\label{fig:time_horizon}]{
\includegraphics[width=0.49\linewidth]{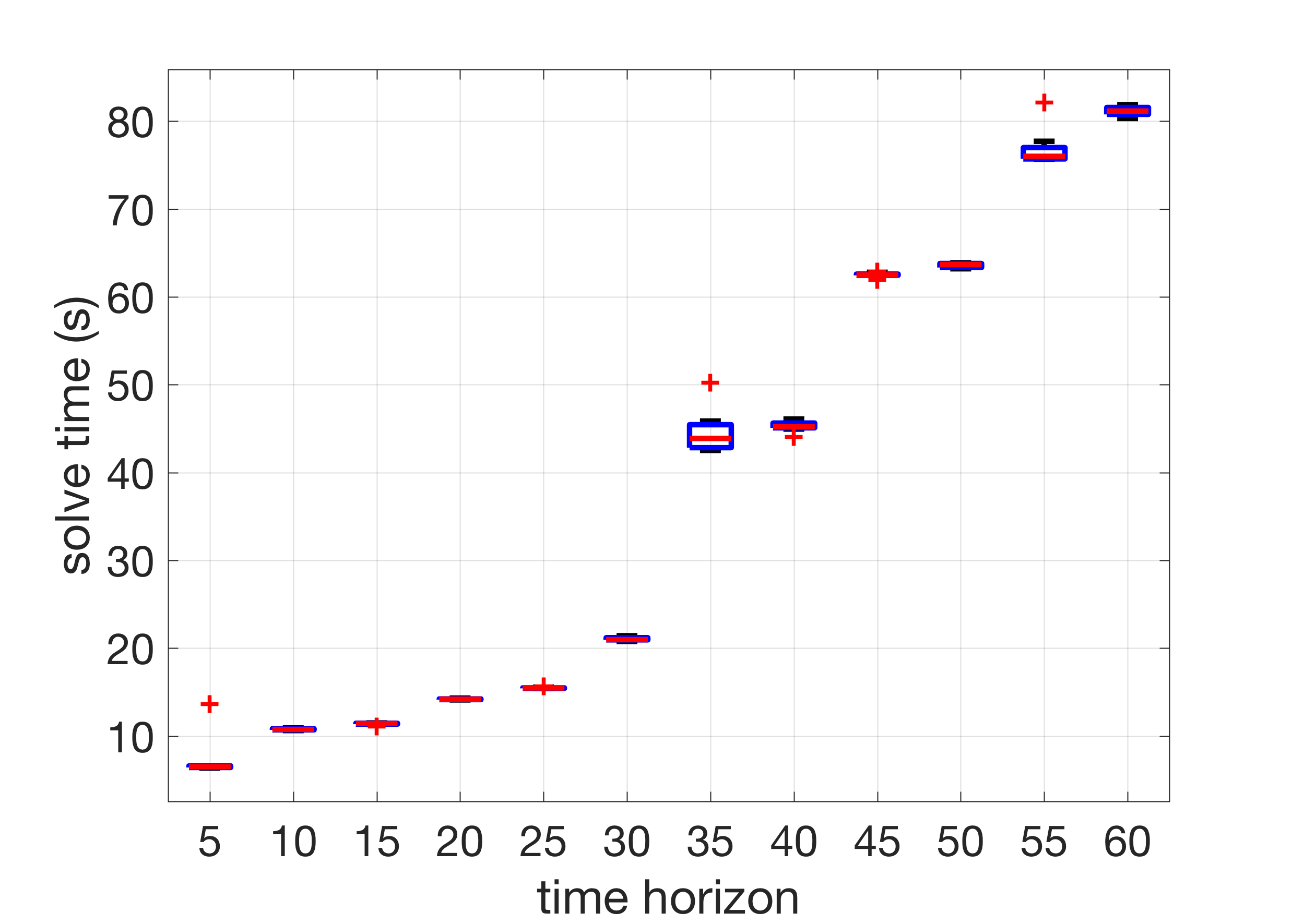}}
\caption{(a) Solve time for the full-scale NLP for different prediction horizons. For prediction horizons $>6$ time-steps, the solver did not converge. (b) SOCP solve time vs. length of prediction horizon.}
\end{figure}

\section{Conclusions and future work}\label{conclusion}
This paper presented a method for the optimal dispatch of batteries in an unbalanced three-phase distribution network. A second order cone relaxation is used to convert the non-convex power flow equation into a convex formulation that can be solved in polynomial time. As the solution obtained from the relaxed problem may not be feasible, an NLP is solved at each time-step  by fixing the real power set-points and decoupling the time-steps to obtain a physically realizable solution. Furthermore, the phenomenon of simultaneous charging and discharging of batteries is analyzed and sufficient conditions are provided for different objective functions that provably avoid this phenomenon to obtain a feasible solution. Simulation tests are conducted on IEEE-13 and IEEE-123 node distribution test feeders showing the feasibility of the obtained solution. The optimality gap is found to be within 2.1\%. The approach is computationally tractable and solves in less than 45~seconds, which ensures that enough time is available for realistic communication delays. This permits an implementation of the  optimization scheme on the minute timescale.

Future work will focus on reducing the optimality gap by using stronger relaxations of the power flow equations. We will also try and provide guarantees for a feasible solution to the decoupled NLP given an initialized SOCP solution. Extending the work to different grid objectives and including mechanical voltage control devices such as transformers and capacitor banks is another scope for improvement. Providing bounds on the gap between voltages obtained from the SOCP solver and a power flow solution is also an avenue for future work. \added{Further analysis on the phenomenon of simultaneous charging and discharging is required as described in Section~\ref{comp_const}}

\appendices
\section{Proof of Theorem \ref{Th1}}\label{AppendixA}
\begin{proof}
Since the SOCP optimization problem is convex and Slater's condition holds trivially, the KKT optimality conditions are both necessary and sufficient. Thus, for the KKT conditions, let
\begin{itemize}
\item  $\mathcal{L}$ be the Lagrangian.
\item $\lambda_{\text{p}}\in \mathbb{R}$ be the Lagrange multiplier for~\eqref{eq:P1_node_real_balance}.
\item $\lambda_{\text{s}}\in \mathbb{R}_+$ be the Lagrange multiplier for inequality~\eqref{eq:P1_battery_inv_limit}.
\item  $\underline{\lambda_{\text{d}}}, \overline{\lambda_{\text{d}}} \in \mathbb{R}_+$ be Lagrange multipliers associated with the lower bound and upper bound of  inequality~\eqref{eq:P1_Pd_limit}, respectively.
\item $\underline{\lambda_{\text{c}}}, \overline{\lambda_{\text{c}}} \in \mathbb{R}_+$ be Lagrange multipliers for the lower and upper bounds of inequality~\eqref{eq:P1_Pc_limit}, respectively.
\end{itemize}
Note that $P^{\text{c}}$ and $P^{\text{d}}$ are the charging and discharging rates for the battery at node $n$, phase $\phi$ at time $t$ and represent primal variables and $\eta_{\text{c}}, \eta_{\text{d}} \in [0,1]$ are the charging and discharging efficiencies.

From the KKT optimality conditions, the following relation is obtained from the Lagrangian with respect to $P^{\text{c}}$, i.e., $\frac{\partial \mathcal{L}}{\partial P^{\text{c}}}\equiv 0$:
\begin{equation}\label{eq:Th1_KKTPc}
\frac{\partial f(x)}{\partial P^{\text{c}}}-\underline{\lambda_{\text{c}}}+\overline{\lambda_{\text{c}}}-\eta_{\text{c}}\Gamma(t) \Delta t
+\lambda_{\text{p}}-2\lambda_{\text{s}}(P^{\text{d}}-P^{\text{c}})=0.
\end{equation}
With respect to $P^{\text{d}}$, KKT conditions give $\frac{\partial \mathcal{L}}{\partial P^{\text{d}}}\equiv 0$:
\begin{multline}\label{eq:Th1_KKTPd}
\frac{\partial f(x)}{\partial P^{\text{d}}}+\alpha (\frac{1}{\eta_{\text{d}}}-\eta_{\text{c}})-\underline{\lambda_{\text{d}}}+\overline{\lambda_{\text{d}}}+\frac{\Gamma(t) \Delta t}{\eta_{\text{d}}}\\-\lambda_{\text{p}}+2\lambda_{\text{s}}(P^{\text{d}}-P^{\text{c}})=0.
\end{multline}
Adding~\eqref{eq:Th1_KKTPc} and~\eqref{eq:Th1_KKTPd} gives:
\begin{equation}\label{eq:Th1_KKTcomb}
\underline{\lambda_{\text{c}}}+\underline{\lambda_{\text{d}}} = \overline{\lambda_{\text{c}}}+\overline{\lambda_{\text{d}}}+\left(\alpha+\Gamma(t) \Delta t\right)\left(\frac{1}{\eta_{\text{d}}}-\eta_{\text{c}}\right)+\frac{\partial f(x)}{\partial P^{\text{c}}}+\frac{\partial f(x)}{\partial P^{\text{d}}}
\end{equation}
In order to avoid SCD, the right hand side of equation~\eqref{eq:Th1_KKTcomb} needs to be strictly positive. In the above equation $\overline{\lambda_{\text{c}}}\geq 0$ and $\overline{\lambda_{\text{d}}}\geq 0$, \added{which changes \eqref{eq:Th1_KKTcomb} to the following inequality:
\begin{equation}\label{eq:Th1_cond2}
   \underline{\lambda_{\text{c}}}+\underline{\lambda_{\text{d}}} \geq \left(\alpha+\Gamma(t) \Delta t\right)\left(\frac{1}{\eta_{\text{d}}}-\eta_{\text{c}}\right)+\frac{\partial f(x)}{\partial P^{\text{c}}}+\frac{\partial f(x)}{\partial P^{\text{d}}} 
\end{equation}}

It can be seen that condition C1 is satisfied by the given objective functions, e.g. for objective $(P^{\text{d}}-P^{\text{c}}-P^{\text{ref}})^2$, $\frac{\partial f(x)}{\partial P^{\text{c}}}+\frac{\partial f(x)}{\partial P^{\text{d}}}=-2(P^{\text{d}}-P^{\text{c}}-P^{\text{ref}})+2(P^{\text{d}}-P^{\text{c}}-P^{\text{ref}})=0$. Based on these facts, \eqref{eq:Th1_cond2} gives:
\begin{equation}\label{eq:Th1_KKTcomb2}
\underline{\lambda_{\text{c}}}+\underline{\lambda_{\text{d}}}\ge \left(\alpha+\Gamma(t) \Delta t\right)\left(\frac{1}{\eta_{\text{d}}}-\eta_{\text{c}}\right)
\end{equation}
Based on the value of $\Gamma(t)$, the problem is divided into two cases:
\begin{enumerate}[I]
    \item $\Gamma(t)\ge0$:
    This is the case where the battery does not hit its upper capacity limit \added{which makes $\Gamma(t)\ge 0$ and as a result}, the right hand side of \eqref{eq:Th1_KKTcomb2} is strictly positive due to condition C2. Hence, if efficiencies are non-unity, simultaneous charging and discharging is avoided in this case. If efficiencies are unity, SCD fictitious losses are zero, so it is always exact.
 
    \item $\Gamma(t)<0$: this implies that the battery must hit its upper limit of state of charge at least once over the prediction horizon. In this case, the battery may waste energy through SCD in order to lower its state of charge.\added{ Parameter $\alpha$ is  added  to  discourage  SCD  in the  battery. When  battery  is  at  its  lower  limit,  SCD  may occur at optimality to consume more power. The $\alpha$ term acts as a  penalty  to  discourage  SCD.} In this case, $\alpha$ would have to be chosen in such a way that condition C3 is satisfied. Hence, conditions C1, C2, and C3 represent sufficient conditions for avoiding SCD.
\end{enumerate}
\end{proof}

\section{Avoiding SCD when tracking a desired battery state of charge}\label{AppendixB}
For the objective function: $f(B_{n,T})=\left(B_{n,T}-B^{\text{d}}\right)^2$, where $B_{n,T}=B_{n,t}+\Delta t\sum_{\tau=t}^{T-1}\left(\eta_{\text{c},n}P^{\text{c}}_{n,\tau}-\frac{1}{\eta_{\text{d},n}} P^{\text{d}}_{n,\tau}\right )$, Corollary \ref{corollary} provides conditions for the relaxation to be exact. These conditions are more restrictive than the ones required in Theorem~\ref{Th1}.
\begin{corollary}\label{corollary}
For the objective function $f(B_{n,T})$, the relaxation is exact under the following conditions: 
\begin{enumerate}
    \item [A1:] $\lambda_{\text{p}} \ge0$.
    \item [A2:] $\alpha>0$
\end{enumerate}
\end{corollary}
\begin{proof}
Let $P^{\text{c}} \ge P^{\text{d}}$, then using KKT conditions, $\underline{\lambda_{\text{c}}}=0$, $\overline{\lambda_{\text{c}}}\ge0$ and the following equation is obtained from the Lagrangian with respect to $P^{\text{c}}$:
\begin{align}\label{eq:cor_Pc1}
 \Gamma(t) \Delta t\ge\frac{1}{\eta_{\text{c}}}(\lambda_{\text{p}}+\frac{\partial f(B_{n,T})}{\partial P^{\text{c}}}-2\lambda_{\text{s}}(P^{\text{d}}-P^{\text{c}}))
\end{align}

Since $\overline{\lambda_{\text{d}}}\ge0$, with respect to $P^{\text{d}}$, the following KKT condition results:
\begin{align}\label{eq:cor_Pd1}
\Gamma(t) \Delta t\le \eta_{\text{d}}(\lambda_{\text{p}}
 -\alpha(\frac{1}{\eta_{\text{d}}}-\eta_{\text{c}})+\underline{\lambda_{\text{d}}}-\frac{\partial f(B_{n,T})}{\partial P^{\text{d}}}-2\lambda_{\text{s}}(P^{\text{d}}-P^{\text{c}}))
\end{align}
Comparing~\eqref{eq:cor_Pc1} and~\eqref{eq:cor_Pd1} gives:
\begin{multline}\label{eq:cor_KKTcomb}
\eta_{\text{d}}\underline{\lambda_{\text{d}}}\ge 
\alpha (1-\eta_{\text{d}}\eta_{\text{c}})+\lambda_{\text{p}}(\frac{1}{\eta_{\text{c}}}-\eta_{\text{d}})+2\lambda_{\text{s}}(\frac{1}{\eta_{\text{c}}}-\eta_{\text{d}})(P^{\text{c}}-P^{\text{d}})
\\+\eta_{\text{d}}\frac{\partial f(B_{n,T})}{\partial P^{\text{d}}}+\frac{1}{\eta_{\text{c}}}\frac{\partial f(B_{n,T})}{\partial P^{\text{c}}}
\end{multline}
Using conditions A1, A2 and the fact that  $\frac{\partial f(B_{n,T})}{\partial P^{\text{c}}}=2\eta_{\text{c}}(B_{n,T}-B^{\text{d}})$, $\frac{\partial f(B_{n,T})}{\partial P^{\text{d}}}=-\frac{2}{\eta_{\text{d}}}(B_{n,T}-B^{\text{d}})$ and $\lambda_{\text{s}}\ge0$ in \eqref{eq:cor_KKTcomb} gives $\underline{\lambda_{\text{d}}}>0$ and hence $P^{\text{d}}=0$, provided $\eta_{\text{c}},\eta_{\text{d}}<1$ and $P^{\text{c}}\ge P^{\text{d}}$. 
A similar procedure can be used to show that when $P^{\text{c}}<P^{\text{d}}$, then $P^{\text{c}}=0$. Hence, $P^{\text{d}}P^{\text{c}}\equiv0$ is enforced.
\end{proof}
\ifCLASSOPTIONcaptionsoff
  \newpage
\fi


%
\bibliographystyle{IEEEtran}
\small\bibliography{fix.bib}
\end{document}